\documentclass[12pt]{amsart}
\usepackage{amsmath}
\usepackage{amssymb}
\usepackage[all]{xy}
\usepackage{longtable}
\usepackage[mathscr]{eucal}
\usepackage{xcolor}
\usepackage[bookmarks=false]{hyperref}
\usepackage{multirow,bigdelim}
\usepackage{nccmath}
\usepackage{url}
\usepackage{comment}
\usepackage{mathrsfs}
\usepackage{arydshln}

\usepackage{amsmath,amsthm,amssymb,mathtools,amsfonts,amsopn,amscd}
\usepackage{bm}
\usepackage{tikz}
\usepackage{tikz-cd}    
\usepackage{rotating}
\usepackage{commath}    
\usepackage{mathtools}    
\usepackage{graphicx}
\usepackage{etoolbox}
\usepackage{enumitem}
\usepackage{stmaryrd}
\usepackage{hyperref}
\usepackage{mathrsfs}
\usepackage{calc}
\usepackage[normalem]{ulem}
\usepackage[numbered]{bookmark}  

\setlist[enumerate,1]{label=(\roman*)}

\usepackage[
backend=biber,
style=alphabetic,
maxalphanames=10,
maxbibnames=10,
sorting=anyt,
]{biblatex}

\theoremstyle{plain}
\newtheorem{thm}{Theorem}[section]
\newtheorem{lem}[thm]{Lemma}
\newtheorem{prop}[thm]{Proposition}
\newtheorem{cor}[thm]{Corollary}

\theoremstyle{definition}
\newtheorem{df}[thm]{Definition}

\theoremstyle{remark}
\newtheorem{rem}[thm]{Remark}

\allowdisplaybreaks[1]

\newcommand{\bbN}{\mathbb{N}}
\newcommand{\ZZ}{\mathbb{Z}}
\newcommand{\NN}{\mathbb{N}}

\newcommand{\RR}{\mathbb{R}}
\newcommand{\CC}{\mathbb{C}}

\newcommand{\FF}{\mathbb{F}}
\newcommand{\GG}{\mathbb{G}}

\newcommand{\Fp}{\FF_p}

\newcommand{\bbF}{\mathbb{F}}

\newcommand{\afk}{\mathfrak{a}}
\newcommand{\bfk}{\mathfrak{b}}
\newcommand{\cfk}{\mathfrak{c}}

\newcommand{\Acal}{\mathcal{A}}
\newcommand{\Bcal}{\mathcal{B}}
\newcommand{\Ccal}{\mathcal{C}}

\newcommand{\Ecal}{\mathcal{E}}

\newcommand{\Ical}{\mathcal{I}}

\newcommand{\Ocal}{\mathcal{O}}

\newcommand{\Rcal}{\mathcal{R}}
\newcommand{\Scal}{\mathcal{S}}
\newcommand{\Tcal}{\mathcal{T}}

\newcommand{\Ycal}{\mathcal{Y}}
\newcommand{\Zcal}{\mathcal{Z}}

\newcommand{\End}{\operatorname{End}}

\newcommand{\id}{\operatorname{id}}

\newcommand{\GL}{\operatorname{GL}}

\newcommand{\Spec}{\operatorname{Spec}}

\newcommand{\ord}{\operatorname{ord}}

\newcommand{\sbe}{\subseteq}

\let\oldforall\forall
\renewcommand{\forall}{\oldforall \: }
\let\oldexist\exists
\renewcommand{\exists}{\oldexist \: }

\let\emptyset\varnothing

\usepackage{stmaryrd}
\SetSymbolFont{stmry}{bold}{U}{stmry}{m}{n}

\newcommand{\dep}{\operatorname{dep}}
\newcommand{\wt}{\operatorname{wt}}

\newcommand{\ww}[1]{\underline{#1}} 

\newcommand{\bfe}{\widehat{e}}
\newcommand{\bfE}{\widehat{E}}

\newcommand{\bff}{\bm{f}}
\newcommand{\bfg}{\bm{g}}

\newcommand{\bfz}{\mathbf{z}}

\newcommand{\Lamz}{\Lambda_{\bfz'}}

\DeclareMathOperator{\Span}{span}
\newcommand{\powerseries}[1]{[\![{#1}]\!]} 

\allowdisplaybreaks

\newcommand{\ang}[1]{\langle #1 \rangle}

\makeatletter
\newcommand{\xequal}[2][]{\ext@arrow 0055{\equalfill@}{#1}{#2}}
\def\equalfill@{\arrowfill@\Relbar\Relbar\Relbar}
\makeatother

\title[Algebra Structures of Multiple Eisenstein Series]{Algebra Structures of Multiple Eisenstein Series in Positive Characteristic}
\author{Ting-Wei Chang, Song-Yun Chen, Fei-Jun Huang, Hung-Chun Tsui}
\date{\today}

\address{}
\email{}

\thanks{The first author is supported by the National Science and Technology Council grant no. 113-2628-M-007-003.
The latter three authors are supported by the National Science and Technology Council grant no. 113-2628-M-007-004.}



\begin{document}

\begin{abstract}
    In \cite{CCHT2025}, the authors introduced multiple Eisenstein series of arbitrary rank in positive characteristic and the $q$-shuffle algebra $\mathcal{E}$ associated with them.
    In the present paper, we establish a class of linear independence results for multiple Eisenstein series.
    We also prove that the $q$-shuffle algebra $\mathcal{R}$ of multiple zeta values embeds into the inverse limit of the spaces of multiple Eisenstein series with respect to the rank $r$, and that $\mathcal{E}$ is isomorphic to the tensor square of $\Rcal$.
    As an application, we show that $\Ecal$ is an associative algebra, thereby verifying the conjecture proposed in \cite{CCHT2025}.
\end{abstract}

\maketitle
\tableofcontents
\section{Introduction}

\subsection{The \texorpdfstring{$q$}{q}-shuffle relation for multiple Eisenstein series}

Fix a prime $p$ and let $\FF_p$ be the prime field of order $p$. Let $\FF_q$ be a finite extension of $\FF_p$ containing $q$ elements.
Let $A := \FF_q[\theta]$ denote the polynomial ring in the variable $\theta$ over $\FF_q$ and $K := \mathbb{F}_q(\theta)$ be the field of fractions of $A$. 
We denote by $K_\infty$ the completion of $K$ at the infinite place and by $\mathbb{C}_\infty$ the completion of a fixed algebraic closure of $K_\infty$. 
We also let $A_+$ be the set of all monic polynomials in $A$.

Let $\mathbb{N}$ be the set of positive integers. For $r \in \mathbb{N}$, let $\Omega^r$ denote the Drinfeld symmetric space of rank $r$, defined as the set of points $[z_1 : \cdots : z_r] \in \mathbb{P}^{r-1}(\mathbb{C}_\infty)$ whose coordinates are $K_\infty$-linearly independent.
As established in \cite[Proposition 6.1]{Drinfeld1974}, $\Omega^r$ carries a rigid analytic structure.
Under the normalization $z_r=1$, we view $\Omega^r$ as the subset of $\mathbb{C}_\infty^r$ given by
\[
\Omega^r = \{ (z_1, \ldots, z_r) \in \mathbb{C}_\infty^r : z_1, \dots, z_{r}, \text{ are } K_\infty\text{-linearly independent and }z_r=1 \}.
\]
We denote by $\Ocal(\Omega^r)$ the algebra of rigid analytic functions on $\Omega^r$.

We now recall the partial order on $A$-lattices introduced in \cite[Definition 1.3]{CCHT2025}, which generalizes the definition provided in \cite{Chen2017} for rank two case.
For $\bff = (f_1,\ldots,f_r)\in A^r$ and $\bfz = (z_1,\ldots,z_r = 1) \in \Omega^r$, we put
\[
\ang{\bff,\bfz} := \sum_{i=1}^r f_iz_i \in \CC_\infty.
\]

\begin{df}
    Let $r\geq 1$, $\bfz = (z_1,\ldots,z_{r-1},z_r=1) \in \Omega^r$ and  $\bff,\bfg\in A^r$.
    Write $\bff=(f_1,\ldots,f_r)$ and $\bfg=(g_1,\ldots,g_r)$.
    We define a partial order $\succ$ on $Az_1+ \cdots +Az_{r-1}+A$ as follows:
    \begin{enumerate}
        \item $\ang{\bff,\bfz}\succ 0$ if $f_1=\cdots=f_{i-1}=0$ and $f_i\in A_+$ for some $1\leq i\leq r$.
        \item For $\ang{\bff,\bfz}, \ang{\bfg,\bfz}\succ 0$, we write $\ang{\bff,\bfz} \succ \ang{\bfg,\bfz}$ if one of the following holds:
        \begin{enumerate}
            \item There exists $1\leq i\leq r$ such that
            $f_1=\cdots=f_{i-1}=g_1=\cdots=g_{i-1}=0$ and
            $f_i,g_i\in A_+$ with $\deg f_i>\deg g_i$.
            \item There exists $1\leq i\leq r$ such that
            $f_1=\cdots=f_{i-1}=g_1=\cdots=g_{i-1}=0$ and
            $f_i\in A_+,g_i=0$.
        \end{enumerate}
    \end{enumerate}
\end{df}

For convenience, let
\[
\Ical:=\bigcup_{m=1}^\infty \NN^m\cup\{\varnothing\}
\]
denote the set of all indices. For a non-empty index $\ww{a}=(a_1,\ldots,a_m)\in\Ical$, we define its \textit{depth} and \textit{weight} by
\[
\dep(\ww{a})=m\quad \text{and} \quad \wt(\ww{a})=a_1+\cdots+a_m,
\]
respectively.
For the empty index $\varnothing$, we put $\dep(\varnothing)=0=\wt(\varnothing)$.
For each $w\in\ZZ_{\geq 0}$, we further denote the subset of indices whose weight is at most $w$ by
\[
\Ical_{\leq w}=\{\ww{a}\in \Ical:\wt(\ww{a})\leq w\}.
\]

We now recall the rank $r$ multiple Eisenstein series introduced in \cite[Definition 1.4]{CCHT2025}, which serve as a generalization of the single and double Eisenstein series on $\Omega^2$ defined by Chen \cite{Chen2017}. It is also worth noting that Pellarin \cite{Pel2025} previously proposed vector-valued versions of multiple Eisenstein series in the rank two case.

\begin{df}[Multiple Eisenstein series] \label{df-MES-positive-charcteristic}
    Let $r\geq 1$. 
    For any non-empty index $\ww{a}=(a_1,\ldots,a_m)\in\Ical$, we define the \textit{rank $r$ multiple Eisenstein series} on $\Omega^r$ by 
    \[
    E_r(\ww{a};\bfz)
    := \sum_{\substack{\bff_1,\ldots,\bff_m \in A^r \\ \ang{\bff_1,\bfz} \succ \cdots \succ \ang{\bff_m,\bfz} \succ 0}}
    \frac{1}{\ang{\bff_1,\bfz}^{a_1}\cdots \ang{\bff_m,\bfz}^{a_m}}.
    \]
    By convention, we put   $E_r(\emptyset;\mathbf{z}):=1$.
\end{df}

It is shown in \cite[Proposition 2.1]{CCHT2025} that the rank $r$ multiple Eisenstein series are rigid analytic functions on $\Omega^r$, i.e., $E_r(\ww{a};\bfz) \in \Ocal(\Omega^r)$. 
In addition, we remark that when $r=1$, they become Thakur's multiple zeta values \cite{thakur2004function}.

It is well-known from \cite{thakur2010shuffle} that the product of Thakur's multiple zeta values can be written as an $\FF_p$-linear combination of multiple zeta values of the same weight, commonly known as the \textit{$q$-shuffle relation}.
The explicit formula is given by the $q$-shuffle algebra $\mathcal{R}$, whose algebra structure was first observed by Yamamoto based on an explicit $q$-shuffle formula of Chen \cite[Theorem 3.1]{Chen2015}, and was subsequently formalized by Shi in her PhD thesis \cite{Shi2018}. This $q$-shuffle algebra $\Rcal$ is defined as follows:

\begin{df}   \label{df-zeta-values-as-words}
    Let $\Scal$ be the free monoid generated by the set $\{x_k : k \in \NN\}$ and let $\Rcal$ be the $\FF_p$-vector space generated by $\Scal$. For a non-empty index $\ww{a}=(a_1,\ldots,a_m)\in\Ical$, we define $x_{\ww{a}}:=x_{a_1}\cdots x_{a_m}$, and for $\ww{a}=\emptyset$, we define $x_\emptyset:=1$.
    An element $x_{\ww{a}}$ is called a \textit{word of depth $m$}.
    When $m \geq 2$, we put $x_{\ww{a}^-}:=x_{a_2}\cdots x_{a_m}$, and when $m=1$, we put $x_{\ww{a}^-}:=x_{\emptyset} = 1$.
    We define the \textit{$q$-shuffle product} $\ast$ on $\Rcal$ inductively on the sum of depths as follows:
    \begin{enumerate}
        \item For the empty word $x_\emptyset=1$ and any $\afk\in\Rcal$, define
        \[
        1*\afk=\afk*1=\afk.
        \]
        \item For non-empty indices $\ww{a}=(a_1,\ldots,a_m)\in\Ical$ and $\ww{b}=(b_1,\ldots,b_n)\in\Ical$, define
        \begin{multline*}
            x_{\ww{a}} * x_{\ww{b}}
            = x_{a_1} ( x_{\ww{a}^-} * x_{\ww{b}} ) 
            + x_{b_1} ( x_{\ww{a}} * x_{\ww{b}^-} ) 
            + x_{a_1 + b_1} ( x_{\ww{a}^-} * x_{\ww{b}^-} ) \\
            + \sum_{\substack{i+j = a_1 + b_1 \\ q-1 \mid  j}} 
            \Delta^{i,j}_{a_1,b_1} x_i ( ( x_{\ww{a}^-} * x_{\ww{b}^-} ) * x_j )
        \end{multline*}
        where
        \begin{enumerate}
            \item $\Delta^{i,j}_{a,b} := (-1)^{a-1}\binom{j-1}{a-1}+(-1)^{b-1}\binom{j-1}{b-1}\in\FF_p$.
            \item $x \sum_i \epsilon_ix_i := \sum_i \epsilon_i (xx_i)$ for $\epsilon_i \in \Fp$ and $x,x_i \in \Scal$.
        \end{enumerate}
        \item Expand the product $\ast$ to the $\Fp$-vector space $\Rcal$ by the distributive law.
    \end{enumerate}
\end{df}

In this way, $\Rcal$ is an $\FF_p$-algebra with respect to the $q$-shuffle product $\ast$.
We mention that throughout this paper, algebras are not necessarily assumed to be associative.

In \cite{CCHT2025}, we showed that the multiple Eisenstein series satisfy the same $q$-shuffle relations as Thakur's multiple zeta values do, thereby extending Shi's result \cite{Shi2018} for the rank one case to arbitrary rank. More precisely, we have the following theorem:

\begin{thm}[{\cite[Theorem 1.6]{CCHT2025}}]\label{thm-E-hat}
    For $r\ge 1$, we define $\bfE_r: \Rcal \to \Ocal(\Omega^r)$ to be the unique $\FF_p$-linear map satisfying  
    \[
    \bfE_r(1) := 1
    \quad
    \text{and}
    \quad
    \bfE_r(x_{\ww{a}}) := E_r(\ww{a};\bfz).
    \]
    Then $\bfE_r$ is an $\FF_p$-algebra homomorphism, i.e.,
    \[
    \bfE_r(x_{\ww{a}} \ast x_{\ww{b}}) = E_r(\ww{a};\bfz) E_r(\ww{b};\bfz).
    \]
\end{thm}

\subsection{Main results}

With the necessary notions introduced, we are now ready to state the main results of this paper on the space of multiple Eisenstein series and the algebra structure of the $q$-shuffle algebra $\Ecal$ as defined in \cite[Definition 3.1]{CCHT2025} (see Definition~\ref{df-Ecal} also).

\begin{df}\label{df-space-of-real-Rcal}
    Let $L \sbe \CC_\infty$ be an $\FF_p$-subalgebra.
    For $r\geq 1$ and $w\in\ZZ_{\geq 0}$, we let 
    \[
    \mathcal{Z}^{(r)}:=\Span_{\FF_p}\{E_r(\ww{a};\bfz):\ww{a}\in\Ical\}
    \quad
    \text{and}
    \quad
    \mathcal{Z}_{L}^{(r)}:=\Span_{L}\{E_r(\ww{a};\bfz):\ww{a}\in\Ical\}
    \] 
    be the $\FF_p$-vector space and the $L$-vector space spanned by all multiple Eisenstein series $E_r(\ww{a}; \bfz)$ of rank $r$, respectively.
\end{df}

We mention that by Theorem \ref{thm-E-hat}, $\Zcal_L^{(r)}$ forms an $L$-algebra for any $\FF_p$-subalgebra $L\subseteq\CC_\infty$.

The following theorem is obtained by analyzing the $t_{\Lambda_{\bfz'}}$-expansion together with the linear independence of multiple Eisenstein series (see \S~\ref{sec-t-expansion-of-MES}).

\begin{thm}[restated as Proposition~\ref{prop-taking-constant-map} and Theorem~\ref{thm-Rcal-inverse-limit}] \label{thm-main-thm1}
    Let $L\subseteq \CC_\infty$ be an $\Fp$-subalgebra and $r\geq 1$.
    \begin{enumerate}
    \item  There is a unique $L$-algebra homomorphism
    \[
    \pi_{r+1,L}:\mathcal{Z}^{(r+1)}_L\to \mathcal{Z}^{(r)}_L
    \]
    such that for all $\ww{a}\in\Ical$,
    \[
    \pi_{r+1,L}(E_{r+1}(\ww{a};\bfz))=E_r(\ww{a};\bfz).
    \]

    \item The realization map $\bfE_r : \Rcal \to \Zcal^{(r)} \subseteq \Ocal(\Omega^r)$ induces an injective $L$-algebra homomorphism
    \[
    \bfE_L:L\otimes_{\Fp} \Rcal \hookrightarrow \varprojlim_{r\in\NN} \Zcal_L^{(r)},
    \quad
    c\otimes x_{\ww{a}} \mapsto (cE_{r}(\ww{a};\bfz))_{r\in\NN},
    \]
    with respect to the inverse system
    $\pi_{r+1,L}:\mathcal{Z}^{(r+1)}_L\to \mathcal{Z}^{(r)}_L$.
    \end{enumerate}
\end{thm}

In particular, Theorem~\ref{thm-main-thm1} (i) follows from the fact that the constant term of the $t_{\Lambda_{\bfz'}}$-expansion of the rank $r+1$ multiple Eisenstein series $E_{r+1}(\ww{a}; \bfz)$ is precisely $E_{r}(\ww{a}; \bfz)$.
On the other hand, as an application of Theorem~\ref{thm-main-thm1} (ii), we obtain the following corollary, which was proposed in \cite[Conjecture 3.2.2]{Shi2018}.

\begin{cor}[restated as Corollary~\ref{cor-assoc-of-R}]
    $(\mathcal{R},\ast)$ forms a commutative associative $\mathbb{F}_p$-algebra.
\end{cor}
\begin{rem}
    The associativity of $(\Rcal,\ast)$ was announced by Im--Kim--Le--Ngo Dac--Pham in \cite{ikldp2023hopf}.
    In their work, the coefficients $\Delta_{a,b}^{i,j}$ arising in the $q$-shuffle product of multiple zeta values are investigated.
    However, our approach to associativity comes from the theory of multiple Eisenstein series and is completely different from theirs.
\end{rem}

To state our next result concerning the algebra structure of $\mathcal{E}$ (see Definition~\ref{df-Ecal}), we first establish some necessary notation.

\begin{df} \label{df-indices}
    \begin{enumerate}
        \item For any non-empty index $\ww{a}=(a_1,\ldots,a_m)\in\Ical$ and $0\leq i\leq m-1$, we define
        \[
        \ww{a}^{(i)} := (a_{i+1},\ldots,a_m)
        \quad
        \text{and}
        \quad
        \ww{a}^{(m)} := \emptyset.
        \]
        On the other hand, for each $1\le i\le m$, we define
        \[
        \ww{a}_{(i)} := (a_1,\ldots,a_i)
        \quad
        \text{and}
        \quad
        \ww{a}_{(0)} := \emptyset.
        \]
        By convention, we put
        \[
        \emptyset^{(i)} = \emptyset_{(i)} := \emptyset
        \]
        for all $i\geq 0$.

        \item For $r\ge 2$ and $\bfz = (z_1,\ldots,z_{r-1},z_r=1)\in \Omega^r$, let $\bfz' : = (z_2,\ldots,z_r)\in \Omega^{r-1}$.
    \end{enumerate}
\end{df}

\begin{rem}\label{rem-first-entry-convention}
    For any non-empty indices $\ww{a}, \ww{b}, \ww{c} \in \mathcal{I}$, we denote their first entries by $a_1, b_1, c_1$, respectively, provided there is no danger of confusion.
\end{rem}

\begin{df}\label{df-Ecal}
    Let $\Tcal$ be the free monoid generated by the set $\{x_k, y_\ell :  k, \ell\in \NN\}$, subject to the commutativity relations $x_k y_\ell = y_\ell x_k$ for all $k,\ell\in \NN$,
    and let $\Ecal$ be the $\Fp$-vector space generated by $\Tcal$.
    For a non-empty index $\ww{a} = (a_1,\ldots, a_m)\in\Ical$, we define $x_{\ww{a}} := x_{a_1}\cdots x_{a_m}$ and $y_{\ww{a}} := y_{a_1}\cdots y_{a_m}$. For $\ww{a}=\emptyset$, we define $x_{\emptyset} = y_{\emptyset} := 1$.
    An element $x_{\ww{a}}$ or $y_{\ww{a}}$ is called a \textit{word of depth $m$}.
    We define the \textit{$q$-shuffle product} $\ast$ on $\Ecal$ inductively on the sum of depths as follows:
    \begin{enumerate}
        \item For the empty word $x_\emptyset = y_\emptyset = 1$ and any $\afk \in \Ecal$, define 
        \[
        1 \ast \afk=\afk \ast 1 = \afk.
        \] 
        
        \item For any indices $\ww{a},\ww{b}\in\Ical$, define 
        \[
        x_{\ww{a}} \ast y_{\ww{b}}= x_{\ww{a}} y_{\ww{b}}= y_{\ww{b}} x_{\ww{a}} = y_{\ww{b}} \ast x_{\ww{a}}.
        \]
        
        \item For non-empty indices $\ww{a} = (a_1,\ldots, a_m) \in \Ical$ and $\ww{b} = (b_1,\ldots, b_n)\in \Ical$, define
        \begin{multline*}
            x_{\ww{a}} \ast x_{\ww{b}}
            = x_{a_1} ( x_{\ww{a}^{(1)}} \ast x_{\ww{b}} ) 
            + x_{b_1} ( x_{\ww{a}} \ast x_{\ww{b}^{(1)}} ) 
            + x_{a_1 + b_1} ( x_{\ww{a}^{(1)}} \ast x_{\ww{b}^{(1)}} ) \\
            + \sum_{\substack{i+j = a_1 + b_1 \\ q-1 \mid  j}} 
            \Delta^{i,j}_{a_1, b_1} x_i ( ( x_{\ww{a}^{(1)}} \ast x_{\ww{b}^{(1)}} ) * x_j ).
        \end{multline*}
        
        \item For non-empty indices $\ww{a} = (a_1,\ldots, a_m) \in \Ical$ and $\ww{b} = (b_1,\ldots, b_n)\in \Ical$, define  
        \begin{multline*}
             y_{\ww{a}} \ast y_{\ww{b}}
            = y_{a_1} ( y_{\ww{a}^{(1)}} \ast y_{\ww{b}} ) 
            + y_{b_1} ( y_{\ww{a}} \ast y_{\ww{b}^{(1)}} ) 
            + y_{a_1+b_1} ( y_{\ww{a}^{(1)}} \ast y_{\ww{b}^{(1)}} ) \\
            + \sum_{\substack{i+j = a_1 + b_1 \\ q-1 \mid  j}} 
            \Delta^{i,j}_{a_1,b_1} y_i ( ( y_{\ww{a}^{(1)}} \ast y_{\ww{b}^{(1)}} ) \ast x_j )
            + \sum_{\substack{i+j = a_1 + b_1 \\ q-1 \mid  j}} 
            \Delta^{i,j}_{a_1,b_1} y_i ( ( y_{\ww{a}^{(1)}} \ast y_{\ww{b}^{(1)}} ) * y_j ).
        \end{multline*}
        
        \item For any indices $\ww{a},\ww{a}',\ww{b},\ww{b}'\in\Ical$, define
        \[
        (x_{\ww{a}} y_{\ww{b}}) \ast (x_{\ww{a}'}y_{\ww{b}'}) 
        = (x_{\ww{a}}*x_{\ww{a}'}) \ast (y_{\ww{b}}*y_{\ww{b}'}).
        \]
        
        \item Expand the product $\ast$ to the $\Fp$-vector space $\Ecal$ by the distributive law.
    \end{enumerate}
\end{df}
We remark that $(\Ecal,\ast)$ forms a commutative $\FF_p$-algebra. Moreover, we observe that $(\Rcal,\ast)$ forms an $\FF_p$-subalgebra of $(\Ecal,\ast)$, and we denote the inclusion map by $\iota:\Rcal\hookrightarrow \Ecal$.
On the other hand, we have another natural map $\bfe:\Rcal\to \Ecal$ arising from the Goss expansion of multiple Eisenstein series (see Proposition \ref{prop-expansion-MES}). This map is defined as follows:

\begin{df} \label{defn-e-hat}
    We define $\bfe:\mathcal{R}\to \mathcal{E}$ to be the unique $\mathbb{F}_p$-linear map such that 
    \[
    \bfe(x_{\ww{a}})=\sum_{i=0}^{\dep(\ww{a})}x_{\ww{a}^{(i)}}y_{\ww{a}_{(i)}},\quad \bfe(x_\varnothing)=1, 
    \]
    for any non-empty index $\ww{a}\in\Ical$.
\end{df}

We then arrive at the following theorem concerning the algebra structure on $\Ecal$.

\begin{thm}[restated as Theorems~\ref{thm-assoc-of-E} and \ref{thm-unique-Hopf-alg-structure-extend-R}] \label{thm-main-thm2}
    We have the following:
    \begin{enumerate}
        \item Let $\phi:\mathcal{R}\otimes_{\mathbb{F}_p} \Rcal\to \Ecal$ be the unique $\bbF_p$-linear map such that
        \[
        \phi(x_{\ww{a}}\otimes x_{\ww{b}})=x_{\ww{a}} * \bfe(x_{\ww{b}})
        \]
        Then $\phi$ is an $\mathbb{F}_p$-algebra isomorphism.
        In particular, $(\mathcal{E},\ast)$ is a commutative associative $\mathbb{F}_p$-algebra defined in Definition~\ref{df-Ecal}.
        \item Let $(\Delta,\varepsilon,S)$ be any Hopf $\mathbb{F}_p$-algebra structure on $\mathcal{R}$. 
        There is a unique Hopf $\mathbb{F}_p$-algebra structure $(\Delta^{\mathrm{e}},\varepsilon^{\mathrm{e}},S^{\mathrm{e}})$ on $\mathcal{E}$ such that the inclusion map $\iota:\mathcal{R}\hookrightarrow \mathcal{E}$ and $\bfe:\mathcal{R}\to \mathcal{E}$ are Hopf $\mathbb{F}_p$-algebra homomorphisms.
        Moreover, endow the Hopf $\mathbb{F}_p$-algebra structure on $\mathcal{R}\otimes_{\mathbb{F}_p} \mathcal{R}$ via the structure induced by the product $\Spec \mathcal{R}\times_{\mathbb{F}_p} \Spec \mathcal{R}$ of $\mathbb{F}_p$-group schemes.
        Then the map $\phi$ is an isomorphism of Hopf $\mathbb{F}_p$-algebras.
        In other words, we have
        \[
        \Spec \Rcal \times_{\FF_p} \Spec \Rcal \simeq \Spec \Ecal
        \]
        as $\FF_p$-group schemes.
    \end{enumerate}
\end{thm}

\begin{rem}
    Shi \cite[Conjecture~3.2.11]{Shi2018} proposed a conjectural Hopf algebra structure on $(\Rcal,\ast)$, and a proof of Shi's conjecture was announced by Im--Kim--Le--Ngo Dac--Pham \cite{ikldp2023hopf}.
\end{rem}

We note that Theorem \ref{thm-main-thm2} (i) confirms the associativity of $\mathcal{E}$, resolving a conjecture previously proposed by the authors in \cite[Remark 3.4]{CCHT2025} based on \texttt{SageMath} computations.

In the classical setting, multiple Eisenstein series were first introduced by Gangl, Kaneko, and Zagier~\cite{gkz2006double} in the depth two case, and later generalized to arbitrary depth by Bachmann~\cite{bachmann2012multiple}.
Subsequently, Bachmann and Tasaka~\cite{bt2018double} showed
that multiple Eisenstein series satisfy various linear relations arising from shuffle regularization, known as the \textit{restricted double shuffle relations}.
This result was further generalized in the recent work of Bachmann and Kanno~\cite{bachmann2026relations}.

In contrast, in \S~\ref{subsec-linear-independence} we establish a class of linear independence results for multiple Eisenstein series in positive characteristic, which lead to the main theorems mentioned above.
Our key idea is to consider multiple Eisenstein series of varying ranks simultaneously, for instance through the $q$-shuffle algebra $\Ecal$ and the inverse limit $\varprojlim\limits_{r\in \NN} \Zcal^{(r)}$.
There seems to be no similar notion of such consideration in the existing classical literature.

\section{\texorpdfstring{$t$}{t}-expansion of multiple Eisenstein series and applications} \label{sec-t-expansion-of-MES}

In this section, we first recall the definitions of Goss polynomials, multiple Goss sums, and Goss expansions.
We then compute a lower bound for the vanishing order of the $t_{\Lambda_{\bfz'}}$-expansion of multiple Goss sums.
The main result of this section is to establish a $\CC_\infty$-linear independence result of multiple Eisenstein series.

\subsection{Multiple Goss sums and Goss expansions}

Fix a positive integer $r$.
Given any $\bfz = (z_1,\ldots,z_{r-1},z_r=1)\in \Omega^r$, we let $\Lambda_{\bfz}$ be the $A$-lattice spanned by the coordinates of $\bfz$.
That is,
\[
\Lambda_{\bfz} := Az_1+\cdots+Az_{r-1}+Az_r \subseteq \CC_\infty.
\]
We define
\[
\exp_{\Lambda_\bfz}(X) := X\prod_{\lambda\in \Lambda_\bfz \setminus \{0\}}\left(1-\frac{X}{\lambda}\right)
\]
as the associated exponential function.
Following \cite[Proposition 15.3]{bbp2024drinfeld}, we can write
\[
\exp_{\Lambda_{\bfz}}(X) = \sum_{n=0}^{\infty}\alpha_n(\bfz) X^{q^n},
\]
where each $\alpha_n(\mathbf{z})$ is a rigid analytic function on $\Omega^r$, i.e., $\alpha_n(\bfz)\in \Ocal(\Omega^r)$.
We further put 
\[
t_{\Lambda_{\bfz}}(X) := \frac{1}{\exp_{\Lambda_{\bfz}}(X)} = \sum_{\lambda\in \Lambda_\bfz} \frac{1}{X+\lambda}.
\]

Associated to each lattice in $\CC_\infty$, Goss constructed a family of polynomials known as \textit{Goss polynomials}.

\begin{df}[{\cite{Gos1980}, \cite{gekeler1988coefficients}}] \label{df-Goss-polynomial}
    The \textit{Goss polynomials} are the unique polynomials 
    \[
    G_n^{\Lambda_\bfz}(X) \in \Ocal(\Omega^{r})[X],\quad n \ge 1,
    \]
    defined recursively as follows.
    \begin{enumerate}
        \item $G_1^{\Lambda_\bfz}(X)=X$.
        \item For $n>1$, 
        \[
        G_n^{\Lambda_\bfz}(X)=X\sum_{i=0}^{\lceil \log_qn\rceil-1}\alpha_i(\bfz)G_{n-q^i}^{\Lambda_\bfz}(X),
        \]
        where $\lceil x\rceil := \inf \{n\in \ZZ : n\ge x\}$ for $x\in\RR$.
    \end{enumerate}
\end{df}

Furthermore, Goss polynomials satisfy the following properties (see \cite{Gos1980} and also \cite{gekeler1988coefficients}).

\begin{prop}\label{prop-basic-properties-for-Goss-polynomials}  
    For each positive integer $n\in\NN$, we have the following properties.
    \begin{enumerate}
        \item $G_n^{\Lambda_\bfz} (t_{\Lambda_{\bfz}}(X))=\sum_{\lambda\in \Lambda_\bfz}1/(X+\lambda)^n$.
        \item For $1\le n\le q$, $G_n^{\Lambda_{\bfz}}(X) = X^n$. 
        \item $G_n^{\Lambda_{\bfz}}(X)$ is a monic polynomial of degree $n$.
        \item $X$ divides $G_n^{\Lambda_{\bfz}}(X)$.
    \end{enumerate}
\end{prop}

Recall that when $r\geq 2$, we set $\bfz'=(z_2,\ldots,z_r=1)\in\Omega^{r-1}$ for $\mathbf{z}=(z_1,\ldots,z_r=1)\in \Omega^r$. 
As before, $\Ical$ denotes the set of all indices, including the empty index. 

\begin{df} [{\cite[Definition 2.4]{CCHT2025}}]  \label{df-Goss-sum}
    Let $r\geq 2$. For any non-empty index $\ww{a} = (a_1, \ldots, a_m) \in \Ical$, we define the \textit{rank $r$ multiple Goss sum}
    \[
    G_r (\ww{a} ; \bfz) 
    := \sum_{\substack{f_1, \ldots ,f_m \in A_+ \\ \deg f_1>\cdots>\deg f_m \ge 0}} 
    G^{\Lamz}_{a_1} (t_{\Lamz}(f_1z_1)) \cdots G^{\Lamz}_{a_m} (t_{\Lamz}(f_mz_1)).
    \]
    By convention, we put $G_r(\emptyset;\mathbf{z})=1$.
\end{df}

Recall the multiple Eisenstein series  $E_r(\ww{a};\mathbf{z})$ defined in Definition \ref{df-MES-positive-charcteristic}.
Multiple Goss sums arise naturally in the study of the expansion of $E_r(\ww{a};\mathbf{z})$.
More precisely, we have the following \textit{Goss expansion} of multiple Eisenstein series.

\begin{prop}[{\cite[Proposition 2.3]{CCHT2025}}] \label{prop-expansion-MES}
    Let $r\ge 2$.
    For any non-empty index $\ww{a} = (a_1,\ldots,a_m) \in \Ical$ and $\bfz \in \Omega^r$, we have 
    \[
    E_r(\ww{a} ; \bfz) = E_{r-1}(\ww{a}; \bfz') + \sum_{i=1}^{m} E_{r-1}(\ww{a}^{(i)}; \bfz') G_r(\ww{a}_{(i)}; \bfz).
    \]
\end{prop}

\subsection{\texorpdfstring{$t$}{t}-expansion of multiple Goss sums and multiple Eisenstein series} \label{subsec-t-expansion-of-MES-and-MGS}

For each $r\geq 1$, we let 
\[
\Gamma_r
:= \left\{
\left(\begin{array}{c|c}
    1 & * \\
    \hline
    0 & I_{r-1}
\end{array}\right)
\in \GL_r(A)
\right\}.
\]
Then $\Gamma_r$ acts on $\mathcal{O}(\Omega^r)$ via 
\[
(\gamma \cdot F)(\bfz) := F(\bfz\cdot\gamma^{\mathrm{tr}}) \quad \text{for} \quad\gamma\in\Gamma_r,
\]
where $\bfz\cdot \gamma^{\mathrm{tr}}$ is the matrix multiplication of the row vector $\mathbf{z}$ and the transpose of $\gamma$.
We denote by $\Ocal(\Omega^r)^{\Gamma_r}$ the subspace consisting of all rigid analytic functions on $\Omega^r$ invariant under $\Gamma_r$.

\begin{rem}\label{rem-gamma-invariant}
     The following examples will be frequently used in the subsequent discussions.
    \begin{enumerate}
        \item For each $F\in \Ocal(\Omega^{r-1})$, the function
        \[
        \bfz\mapsto F(\bfz')
        \]
        defines a $\Gamma_r$-invariant rigid analytic function on $\Omega^r$, which is still denoted by $F$ if there is no danger of confusion.
        This induces a embedding into $\mathcal{O}(\Omega^r)^{\Gamma_r} \subseteq \mathcal{O}(\Omega^r)$:
        
        \begin{alignat*}{2}
            \mathcal{O}(\Omega^{r-1}) &\;\hookrightarrow\; && \mathcal{O}(\Omega^r)^{\Gamma_r}\\
            F & \;\mapsto\; && F:=(\mathbf{z}\mapsto F(\mathbf{z}')). 
        \end{alignat*}
        Throughout this paper, we identify $\mathcal{O}(\Omega^{r-1})$ with its image under the map above.

        \item For $f\in A_+$, the function $H:\mathbf{z}\mapsto t_{\Lambda_{\mathbf{z}'}}(fz_1):= t_{\Lambda_{\mathbf{z}'}}(X)|_{X=fz_1}$ is $\Gamma_r$-invariant.
        To verify this, let
        \[
        \gamma=\begin{pmatrix}
            1 & a_2 & \cdots & a_r \\
            0 & 1 & \cdots & 0 \\
            \vdots & \vdots & \ddots & \vdots \\
            0 & 0 & \cdots & 1
        \end{pmatrix}\in \Gamma_r
        \]
        with $a_2,\ldots,a_r\in A$.
        Since $fa_2z_2+\cdots+fa_rz_r\in \Lambda_{\mathbf{z}'}$, we have
        \begin{align*}
            (\gamma \cdot H)(\bfz) 
            &= H(\mathbf{z} \cdot \gamma^{\operatorname{tr}}) \\
            &= H(z_1+a_2z_2+\cdots+a_rz_r,z_2,\ldots,z_r) \\
            &=t_{\Lambda_{\mathbf{z}'}}(fz_1+fa_2z_2+\cdots+fa_rz_r) \\
            &=\frac{1}{\exp_{\Lambda_{\mathbf{z}'}}(fz_1)+\exp_{\Lambda_{\mathbf{z}'}}(fa_2z_2+\cdots+fa_rz_r)} \\
            &=\frac{1}{\exp_{\Lambda_{\mathbf{z}'}}(fz_1)} =t_{\Lambda_{\mathbf{z}'}}(fz_1) = H(\mathbf{z}).
        \end{align*}
    \end{enumerate}
\end{rem}

We next recall the $t_{\Lambda_{\bfz'}}$-expansion of $\Gamma_r$-invariant rigid analytic functions in the arbitrary rank setting developed in~\cite{bbp2024drinfeld} and~\cite{Gek25-DMF-VII}.

\begin{prop}[{\cite[Proposition 5.4]{bbp2024drinfeld}, \cite[(7.12.1)]{Gek25-DMF-VII}}]\label{prop-t-expansion}
    Let $r\geq 2$.
    For each $F \in \Ocal(\Omega^r)^{\Gamma_r}$, there is a unique sequence $\{F_n\}_{n\in \mathbb{Z}}$ in $\mathcal{O}(\Omega^{r-1})$ such that $F$ admits the $t_{\Lambda_{\bfz'}}$-expansion
    \[
    F(\mathbf{z})=\sum_{n=-\infty}^{\infty}F_n(\mathbf{z}')t_{\Lambda_{\mathbf{z}'}}(z_1)^n,
    \]
    valid for $\mathbf{z} = (z_1, \ldots, z_{r-1}, 1)\in \Omega^r$ in a neighborhood of infinity.
\end{prop}

\begin{df}[{\cite[Definition 5.12]{bbp2024drinfeld}}]\label{df.1/t-adic-valution}
    Let $r\ge 2$.
    For each $F\in \mathcal{O}(\Omega^r)^{\Gamma_r}$ with the $t_{\Lambda_{\mathbf{z}'}}$-expansion given in Proposition \ref{prop-t-expansion}, we define
    \[
    \ord_r (F) := \inf \{n \in \ZZ : F_n(\bfz') \neq 0 \}\in \mathbb{Z}\cup \{\pm \infty\}.
    \]
\end{df}

We note that for any $F_1,F_2\in \mathcal{O}(\Omega^r)^{\Gamma_r}$, we have
\begin{equation}\label{eq-ord_r-valuation-multiplicative}
    \ord_r(F_1F_2)=\ord_r(F_1)+\ord_r(F_2)
\end{equation}
and
\begin{equation}\label{eq-ord_r-valuation-additive}
    \ord_r(F_1+F_2)\ge \min \{\ord_r(F_1),\ord_r(F_2)\}.
\end{equation}

To obtain a lower bound for $\ord_r(G_r(\ww{a};\mathbf{z}))$, we require the theory of Drinfeld $A$-modules, originally introduced in \cite{Drinfeld1974}. 
We briefly review the necessary background as follows.

\begin{df}[{\cite[Definition 4.4.2]{goss1996basic}}]
    Let $r\geq 1$. A \textit{rank $r$ Drinfeld $A$-module} $\phi$ over $\CC_\infty$ is an $\FF_q$-algebra homomorphism
    \begin{alignat*}{2}
            \phi:A &\;\rightarrow\; && \End_{\FF_q}(\GG_a(\CC_\infty))\\
            f & \;\mapsto\; && \phi_f:=g_{f,0}\id +g_{f,1}\tau^1+\cdots+g_{f,r\deg f}\tau^{r\deg f} 
        \end{alignat*}
    where for each $f\in A$, we require:
    \begin{enumerate}
        \item $g_{f,i}\in\CC_\infty$ for $0\leq i\leq r\deg f$. 
        \item$g_{f,0}=f$.
        \item $g_{f,r\deg f}\neq 0$.
    \end{enumerate}
    Here, $\tau:x\mapsto x^q$ denotes the $q$-th power Frobenius endomorphism on $\CC_\infty$-valued points of the additive group $\GG_a$.
    Furthermore, we let \[\phi_f(X):=g_{f,0}X +g_{f,1}X^q+\cdots+g_{f,r\deg f}X^{q^{r\deg f}}\in\CC_\infty[X]\]  be the associated $\FF_q$-linear polynomials.
\end{df}
Now, let $\mathbf{z} = (z_1,\ldots,z_{r-1},z_r=1)\in \Omega^r$ and consider the associated lattice
\[
\Lambda_{\mathbf{z}}
=Az_1+Az_2+\cdots+Az_{r-1}+Az_r.
\]
It follows from Drinfeld's uniformization theorem (see \cite[\S~4]{goss1996basic}) that there exists a unique rank $r$ Drinfeld $A$-module $\phi^{\Lambda_{\bfz}}$ over $\mathbb{C}_{\infty}$ such that for all $f\in A$,
\[
\exp_{\Lambda_{\bfz}}(fX)=\phi^{\Lambda_{\bfz}}_f(\exp_{\Lambda_{\bfz}}(X)).
\]
For each $f\in A$, we write 
\[
\phi^{\Lambda_{\bfz}}_f(X)
:= \sum_{i=0}^{r\deg f} g_{f,i} (\bfz) X^{q^i},
\]
and view the coefficient forms $g_{f,i}(\bfz)$ as functions on $\Omega^r$. Then it is known by \cite[Proposition 15.12]{bbp2024drinfeld} that for each $0\leq i\leq r\deg f$, we have 
$g_{f,i}(\bfz) \in \Ocal(\Omega^{r})$ and $g_{f,r\deg f} (\bfz)\in \Ocal(\Omega^{r})^\times$. 
Here, we denote the group of units of $\Ocal(\Omega^r)$ by $\mathcal{O}(\Omega^r)^{\times}$.

In the following lemma, we provide a lower bound for $\ord_r(G_r(\ww{a};\mathbf{z}))$.

\begin{lem}\label{lem-1/t-adic-valuation-of-G(t(fz))}
    Let $r\ge 2$.
    For any index $\ww{a}\in \Ical$, we have that $G_r(\ww{a};\mathbf{z})\in \mathcal{O}(\Omega^r)^{\Gamma_r}$ and 
    \[
    \ord_r(G_r(\ww{a};\mathbf{z}))\ge \frac{q^{(r-1)\dep (\ww{a})}-1}{q^{r-1}-1}.
    \]
\end{lem}

\begin{proof}  
    For the empty index, the result is clear.
    We may assume $\ww{a}=(a_1,\ldots,a_m)\in \Ical$ is non-empty.
    By Remark \ref{rem-gamma-invariant}, the function $\bfz\mapsto t_{\Lambda_{\mathbf{z}'}}(fz_1)$ is $\Gamma_r$-invariant on $\Omega^r$ for each $f \in A_+$.
    Thus, one sees from Definitions~\ref{df-Goss-polynomial} and \ref{df-Goss-sum} that
    \[
    G_r(\ww{a};\mathbf{z})\in \mathcal{O}(\Omega^r)^{\Gamma_r}.
    \]

    To obtain the desired lower bound for $\ord_r (G_r(\ww{a};\bfz))$, by Definition \ref{df-Goss-sum} and \eqref{eq-ord_r-valuation-additive}, it suffices to show that for all $f_1,\ldots,f_m\in A_+$ with $\deg f_1>\cdots>\deg f_m \geq 0$, we have
    \begin{align*}
        \ord_r\left(G_{a_1}^{\Lambda_{\mathbf{z}'}}(t_{\Lambda_{\mathbf{z}'}}(f_1z_1))
        \cdots 
        G_{a_m}^{\Lambda_{\mathbf{z}'}}(t_{\Lambda_{\mathbf{z}'}}(f_mz_1))\right) 
        &\ge \frac{q^{(r-1)m}-1}{q^{r-1}-1}.
    \end{align*}
    For each $\mathbf{z} = (z_1,\ldots,z_{r-1},z_r=1)\in \Omega^r$, we denote by $\phi^{\Lambda_{\bfz'}}$ the Drinfeld $A$-module corresponding to the lattice $\Lambda_{\bfz'}$ and write
    \[
    \phi^{\Lambda_{\bfz'}}_f(X)
    = \sum_{i=0}^{(r-1)\deg f} g_{f,i} (\bfz') X^{q^i}.
    \]
    We note that by the discussion above, $g_{f,i}(\bfz') \in \Ocal(\Omega^{r-1})$ for $0\le i\le (r-1)\deg f$ and $g_{f,(r-1)\deg f} (\bfz')\in \Ocal(\Omega^{r-1})^\times$.
    It follows that for $f\in A_+$
    \begin{align*}
        t_{\Lambda_{\mathbf{z}'}}(fz_1)
        &= \frac{1}{\exp_{\Lambda_{\mathbf{z}'}}(fz_1)}
        = \frac{1}{\phi^{\Lambda_{\mathbf{z}'}}_f(\exp_{\Lambda_{\mathbf{z}'}}(z_1))}
        = \frac{1}{\phi^{\Lambda_{\mathbf{z}'}}_f(t_{\Lambda_{\mathbf{z}'}}(z_1)^{-1})} \\
        &=\left( g_{f,(r-1)\deg f}(\bfz')\cdot t_{\Lambda_{\mathbf{z}'}}(z_1)^{-q^{(r-1)\deg f}} + \sum_{i=0}^{(r-1)\deg f-1} g_{f,i}(\bfz') \cdot t_{\Lambda_{\mathbf{z}'}}(z_1)^{-q^i} \right)^{-1} \\
        &=\frac{t_{\Lambda_{\mathbf{z}'}}(z_1)^{q^{(r-1)\deg f}}}{g_{f,(r-1)\deg f}(\bfz')}\cdot \left(1+\sum_{i=0}^{(r-1)\deg f-1} \frac{g_{f,i}(\bfz')}{g_{f,(r-1)\deg f}(\bfz')} \cdot t_{\Lambda_{\mathbf{z}'}}(z_1)^{q^{(r-1)\deg f}-q^i}\right)^{-1}.
    \end{align*}
    Since
    \[
    1+\sum_{i=0}^{(r-1)\deg f-1} \frac{g_{f,i}(\bfz')}{g_{f,(r-1)\deg f}(\bfz')} \cdot X^{q^{(r-1)\deg f}-q^i}
    \in\Ocal(\Omega^{r-1})[\![X]\!]^\times,
    \]
    we see that
    \[
    \ord_r(t_{\Lambda_{\mathbf{z}'}}(fz_1))=q^{(r-1)\deg f}.
    \]
    By Proposition \ref{prop-basic-properties-for-Goss-polynomials}, for each $a \in \NN$, $X$ divides $G_a^{\Lamz}(X)$, so we have
    \begin{equation}\label{eq-order-of-t(fz)}
        \ord_r \left(G_a^{\Lambda_{\bfz'}} (t_{\Lambda_{\bfz'}}(fz_1))\right)
        \geq q^{(r-1) \deg f}.
    \end{equation}
    By \eqref{eq-ord_r-valuation-multiplicative}, for $f_1,\ldots,f_m\in A_+$ with $\deg f_1>\cdots>\deg f_m \geq 0$, we have
    \begin{align*}
        \ord_r\left(G_{a_1}^{\Lambda_{\mathbf{z}'}}(t_{\Lambda_{\mathbf{z}'}}(f_1z_1))\cdots G_{a_m}^{\Lambda_{\mathbf{z}'}}(t_{\Lambda_{\mathbf{z}'}}(f_mz_1))\right)
        &= \sum_{j=1}^{m} \ord_r\left(G_{a_j}^{\Lambda_{\mathbf{z}'}}(t_{\Lambda_{\mathbf{z}'}}(f_jz))\right) \\
        &\ge \sum_{j=1}^{m}q^{(r-1)\deg f_j}  
        \geq \sum_{j=0}^{m-1}q^{(r-1)j} 
        = \frac{q^{(r-1)m}-1}{q^{r-1}-1},
    \end{align*}
    and hence the proof is completed.
\end{proof}

We now recall that $\Zcal^{(r)}_L\subseteq \Ocal(\Omega^r)$ denotes the $L$-algebra generated by all rank $r$ multiple Eisenstein series for any $\FF_p$-subalgebra $L\subseteq \CC_\infty$ given in Definition \ref{df-space-of-real-Rcal}.
The following proposition enables us to consider the inverse limit of $\Zcal^{(r)}_L$ in \S~\ref{sec-Hopf-algebra-on-E}.

\begin{prop}\label{prop-taking-constant-map}
     For any $\Fp$-subalgebra $L\subseteq \CC_\infty$ and $r\geq 1$, there is a unique $L$-algebra homomorphism
    \[
    \pi_{r+1,L}:\mathcal{Z}^{(r+1)}_L\to \mathcal{Z}^{(r)}_L
    \]
    such that for all $\ww{a}\in\Ical$,
    \[
    \pi_{r+1,L}(E_{r+1}(\ww{a};\mathbf{z}))=E_r(\ww{a};\mathbf{z}).
    \]
\end{prop}
\begin{proof}
    First, we recall from Proposition~\ref{prop-expansion-MES} that for each $r \geq 1$ and any non-empty index $\ww{a}\in \Ical$, we have the following Goss expansion
    \[
    E_{r+1}(\ww{a};\mathbf{z})=E_{r}(\ww{a};\mathbf{z}') +\sum_{i=1}^{\dep(\ww{a})}E_{r}(\ww{a}^{(i)};\mathbf{z}')G_{r+1}(\ww{a}_{(i)};\mathbf{z}).
    \]
    By Remark \ref{rem-gamma-invariant} and Lemma \ref{lem-1/t-adic-valuation-of-G(t(fz))}, we have that $E_r(\ww{a}^{(i)};\mathbf{z}'),G_{r+1}(\ww{a}_{(i)};\mathbf{z})\in \mathcal{O}(\Omega^{r+1})^{\Gamma_{r+1}}$ for all $0\leq i\leq \dep(\ww{a})$.
    It follows that
    \[
    E_{r+1}(\ww{a};\mathbf{z})\in \Ocal(\Omega^{r+1})^{\Gamma_{r+1}}
    \]
    and hence
    \[
    \Zcal_L^{(r+1)} \subseteq \Ocal(\Omega^{r+1})^{\Gamma_{r+1}}.
    \]

    Now, given $F\in \Zcal_L^{(r+1)}\subseteq \mathcal{O}(\Omega^{r+1})^{\Gamma_{r+1}}$, by Proposition \ref{prop-t-expansion}, there is a unique sequence $\{F_n\}_{n\in\ZZ}$ in $\Ocal(\Omega^r)$ such that $F$ admits the $t_{\Lambda_{\bfz'}}$-expansion
    \[
    F(\mathbf{z})=\sum_{n=-\infty}^{\infty}F_n(\mathbf{z}')t_{\Lambda_{\mathbf{z}'}}(z_1)^n.
    \] 
    Define the map $\pi_{r+1,L} : \Zcal^{(r)}_L \to \Ocal(\Omega^r)$ by taking the constant term in the $t_{\Lamz}$-expansion, i.e.,
    \[
    \pi_{r+1,L}(F):=F_0\in \Ocal(\Omega^r).
    \]

    Next, we calculate the constant term of $E_{r+1}(\ww{a};\bfz)$.
    By Lemma~\ref{lem-1/t-adic-valuation-of-G(t(fz))} and \eqref{eq-ord_r-valuation-additive}, we have
    \[
    \ord_{r+1}\left(\sum_{i=1}^{\dep(\ww{a})}E_r(\ww{a}^{(i)};\mathbf{z}')G_{r+1}(\ww{a}_{(i)};\mathbf{z})\right)\geq 1,
    \]
    and hence the constant term of $E_{r+1}(\ww{a};\bfz)$ with respect to the $t_{\Lambda_{\bfz'}}$-expansion is precisely $E_r(\ww{a};\bfz')$.
    That is, 
    \[
    \pi_{r+1,L}(E_{r+1}(\ww{a};\mathbf{z}))=E_r(\ww{a};\mathbf{z}).
    \]

    Finally, given arbitrary $F\in \Zcal_L^{(r+1)}$, by writing $F$ as an $L$-linear combination of rank $(r+1)$ multiple Eisenstein series and applying the uniqueness of $t_{\Lamz}$-expansion, we see that $F_0\in\Zcal_L^{(r)}$.
    The fact that
    \[
    \pi_{r+1,L}:\Zcal_L^{(r+1)}\to \Zcal_L^{(r)}
    \] 
    is an $L$-algebra homomorphism follows immediately from the definition of $\pi_{r+1,L}$ and the uniqueness of $t_{\Lamz}$-expansion.
 \end{proof}

\subsection{Linear independence of multiple Eisenstein series}
\label{subsec-linear-independence}
In this section, we investigate linear independence results of multiple Eisenstein series.
We recall that the set 
\[
\Ical_{\le w}=\{\ww{a}\in \Ical : \wt(\ww{a})\le w\}
\]
consists of all indices of weight at most $w\geq 0$.

In view of the following theorem, any finite set of multiple Eisenstein series is $\CC_\infty$-linearly independent, provided that the rank in question is sufficiently large

\begin{thm}\label{thm-linear-indep-of-MES}
    Suppose that $w\in \ZZ_{\geq 0}$.
    Then for each $r>(w+1)+\log_q(w+1)$, 
    \[
    \{E_r(\ww{a};\mathbf{z}):\ww{a} \in \Ical_{\le w}\}\subseteq \Ocal(\Omega^{r})
    \]
    is a $\mathbb{C}_{\infty}$-linearly independent subset of $\mathcal{O}(\Omega^r)^{\Gamma_r}$.
\end{thm}

\begin{proof}
    We proceed by induction on $w$.
    When $w=0$, the only possible index is the empty index.
    Then for $r > 0+1+\log_q(0+1) = 1$, we have $E_r(\emptyset;\mathbf{z})=1$, which is clearly $\CC_\infty$-linearly independent.
    This proves the base case $w=0$.

    Let $w_0$ be a non-negative integer and assume that the assertion holds for $w= w_0$.
    Now we let $w=w_0+1$ and fix any integer $r>(w+1)+\log_q(w+1)$. 
    Assume that there exists $g_{\ww{a}}\in \mathbb{C}_{\infty}$ such that 
    \begin{equation}\label{eq-linear-combination}
        \sum_{\ww{a} \in \Ical_{\le w}} g_{\ww{a}}E_r(\ww{a};\mathbf{z}) = 0.
    \end{equation}
    
    Let $\ww{a} = (a_1,\ldots, a_m) \in \Ical_{\leq w}$ be any index of $\wt(\ww{a})\leq w$. 
    By Lemma \ref{lem-1/t-adic-valuation-of-G(t(fz))}, for $2\le i\le \dep(\ww{a})$, 
    \begin{equation}\label{eq-G_r-lower-bound1}
        \ord_r(G_r(\ww{a}_{(i)};\mathbf{z}))
    \ge \frac{q^{(r-1)i}-1}{q^{r-1}-1}
    \ge \frac{q^{2(r-1)}-1}{q^{r-1}-1}
    = q^{r-1}+1>q^{r-1},
    \end{equation}
    where $\ww{a}_{(i)}=(a_1,\ldots,a_i)$ is defined in Definition \ref{df-indices}.
    On the other hand,
    \[
     G_r(\ww{a}_{(1)};\mathbf{z})
     = \sum_{f\in A_+} G_{a_1}^{\Lambda_{\mathbf{z}'}}(t_{\Lambda_{\mathbf{z}'}}(fz_1))
     = G_{a_1}^{\Lambda_{\mathbf{z}'}}(t_{\Lambda_{\mathbf{z}'}}(z_1))+\sum_{1\neq f\in A_+}G_{a_1}^{\Lambda_{\mathbf{z}'}}(t_{\Lambda_{\mathbf{z}'}}(fz_1)).
    \]
    By \eqref{eq-order-of-t(fz)}, for $f\in A_+$ and $f\neq 1$, we have
    \[
        \ord_r \left(G_{a_1}^{\Lambda_{\bfz'}} (t_{\Lambda_{\bfz'}}(fz_1))\right)
    \ge q^{(r-1) \deg f}\geq q^{r-1}.
    \]
    It follows from \eqref{eq-ord_r-valuation-additive} that
    \begin{equation}\label{eq-G_r-lower-bound2}
    \ord_r\left(\sum_{1\neq f\in A_+} G_{a_1}^{\Lambda_{\mathbf{z}'}}(t_{\Lambda_{\mathbf{z}'}}(fz_1))\right)
    \ge q^{r-1}.
    \end{equation}
    By Proposition \ref{prop-basic-properties-for-Goss-polynomials} (iii), $G_{a_1}^{\Lambda_{\mathbf{z}'}}(X)$ is monic and
    $$\deg_XG_{a_1}^{\Lambda_{\mathbf{z}'}}(X) = a_1\le \wt(\ww{a})\le w < (w+1)q^w=q^{w+\log_q(w+1)}<q^{r-1},$$ so we have
    \begin{equation}\label{eq-t-expansion-of-G_r}
        G_r(\ww{a}_{(1)};\mathbf{z})= F_1(t_{\Lamz}(z_1))t_{\Lambda_{\bfz'}}(z_1)^{q^{r-1}} + t_{\Lambda_{\mathbf{z}'}}(z_1)^{a_1}+\text{lower degree terms}
    \end{equation}
    for some $F_1(X)\in \mathcal{O}(\Omega^{r-1})\powerseries{X}$.
    Therefore, by Proposition~\ref{prop-expansion-MES} and the computation above, we have
    \begin{align} \label{eq-congruence-of-MES-modulo-t^(q^(r-1))}
        E_r(\ww{a};\mathbf{z}) 
        &= E_{r-1}(\ww{a};\mathbf{z}')+E_{r-1}(\ww{a}^{(1)};\mathbf{z}')G_r(\ww{a}_{(1)};\mathbf{z})+\sum_{i=2}^{\dep(\ww{a})}E_{r-1}(\ww{a}^{(i)};\mathbf{z}')G_r(\ww{a}_{(i)};\mathbf{z}) \\
        &= \begin{multlined}[t]
             E_{r-1}(\ww{a};\mathbf{z}') + E_{r-1}(\ww{a}^{(1)};\bfz')
             \left(G_{a_1}^{\Lambda_{\mathbf{z}'}}(t_{\Lambda_{\mathbf{z}'}}(z_1)) + \sum_{1\neq f\in A_+}G_{a_1}^{\Lambda_{\mathbf{z}'}}(t_{\Lambda_{\mathbf{z}'}}(fz_1))\right) \\
             + \sum_{i=2}^{\dep(\ww{a})}E_{r-1}(\ww{a}^{(i)};\mathbf{z}')G_r(\ww{a}_{(i)};\mathbf{z})
        \end{multlined} \notag\\
        &= \text{lower degree terms} + E_{r-1}(\ww{a}^{(1)};\mathbf{z}')t_{\Lambda_{\mathbf{z}'}}(z_1)^{a_1} + F_2(t_{\Lamz}(z_1))t_{\Lambda_{\bfz'}}(z_1)^{q^{r-1}} \notag
    \end{align}
    for some $F_2(X)\in \mathcal{O}(\Omega^{r-1})\powerseries{X}$.

    Observe that we may write \eqref{eq-linear-combination} as 
    \begin{equation}\label{eq-decompose-linear-relation-MES}
        0=\sum_{\ww{a} \in \Ical_{\le w}} g_{\ww{a}}E_r(\ww{a};\mathbf{z})=\sum_{i=1}^{w}\left(\sum_{\substack{\ww{a} \in \Ical_{\le w} \\ a_1=i}}g_{\ww{a}}E_r(\ww{a};\mathbf{z})\right)+g_{\varnothing}E_r(\varnothing;\bfz).
    \end{equation}
    Notice that the only index $\ww{a}\in\Ical_{\leq w}$ of $\wt(\ww{a}) = w$ with $a_1 = w$ is $(w)$.
    By \eqref{eq-congruence-of-MES-modulo-t^(q^(r-1))}, the coefficient of $t_{\Lambda_{\mathbf{z}'}}(z_1)^w$ in \eqref{eq-decompose-linear-relation-MES} is
    \[
    g_{(w)}E_{r-1}(\emptyset;\mathbf{z}') = g_{(w)},
    \]
    so by the uniqueness part of Proposition \ref{prop-t-expansion}, $g_{(w)}=0$.
    
    Assume that $g_{\ww{a}}=0$ for any index $\ww{a}\in\Ical_{\leq w}$ with $a_1 \geq i$ for some $2\le i\leq w$.
    We claim that $g_{\ww{a}}=0$ for any index $\ww{a}\in\Ical_{\leq w}$ with $a_1 = i-1$.
    By the assumption and \eqref{eq-congruence-of-MES-modulo-t^(q^(r-1))}, we observe that the coefficient of $t_{\Lambda_{\mathbf{z}'}}(z_1)^{i-1}$ in \eqref{eq-decompose-linear-relation-MES} is
    \[
    \sum_{\substack{\ww{a} \in \Ical_{\le w} \\ a_1=i-1}}g_{\ww{a}}E_{r-1}(\ww{a}^{(1)};\mathbf{z}')=0.
    \]
    We note that 
    \[
    \max \{\wt(\ww{a}^{(1)}):\ww{a} \in \Ical_{\le w}, a_1=i-1\}\le w-(i-1) \leq w-1= w_0.
    \]
    Since $r-1>w+\log_q(w+1)> w_0+1+\log_q(w_0+1)$, by induction hypothesis, the set 
    $$\{E_{r-1}(\ww{a}^{(1)};\mathbf{z}'):\ww{a} \in \Ical_{\le w}, a_1=i-1\}$$
    is $\CC_\infty$-linearly independent.
    Thus, we have $g_{\ww{a}}=0$ for any index $\ww{a}\in\Ical_{\leq w}$ with $a_1 = i-1$.
    
    We have shown that $g_{\ww{a}}=0$ for any non-empty index $\ww{a}\in\Ical_{\leq w}$.
    Thus, \eqref{eq-decompose-linear-relation-MES} turns into $g_{\varnothing}E_r(\varnothing;\bfz)=g_{\varnothing}=0$.
    We now conclude that the set 
    \[
    \{E_r(\ww{a};\mathbf{z}):\ww{a}\in \Ical_{\le w}\}=\{E_r(\ww{a};\mathbf{z}):\ww{a}\in \Ical_{\le w_0+1}\}
    \]
    is $\CC_\infty$-linearly independent, and hence the result follows from induction on $w$.
\end{proof}

\begin{prop}\label{prop-O(Ω^r-1)-linear-indep-of-MES}
    Let $r \geq 2$ and $\ww{a}_1, \ldots, \ww{a}_q \in \Ical$ such that for $1\le i\le q$, the first entry of $\ww{a}_i$ is $i$.
    Then  $E_r(\ww{a}_1;\mathbf{z}),\ldots,E_r(\ww{a}_q,\mathbf{z})$ are $\mathcal{O}(\Omega^{r-1})$-linearly independent.
\end{prop}
\begin{proof}
    Given any non-empty index $\ww{b} = (b_1,\ldots,b_m)$ with $b_1\le q$, by Proposition \ref{prop-expansion-MES}, we have 
    \begin{align*}
        E_r(\ww{b};\mathbf{z}) 
        &= E_{r-1}(\ww{b};\mathbf{z}')+E_{r-1}(\ww{b}^{(1)};\mathbf{z}')G_r(\ww{b}_{(1)};\mathbf{z})+\sum_{i=2}^{\dep(\ww{b})}E_{r-1}(\ww{b}^{(i)};\mathbf{z}')G_r(\ww{b}_{(i)};\mathbf{z}) \\
        &= \begin{multlined}[t]
             E_{r-1}(\ww{b};\mathbf{z}') + E_{r-1}(\ww{b}^{(1)};\bfz')
             \left(G_{b_1}^{\Lambda_{\mathbf{z}'}}(t_{\Lambda_{\mathbf{z}'}}(z_1)) + \sum_{1\neq f\in A_+}G_{b_1}^{\Lambda_{\mathbf{z}'}}(t_{\Lambda_{\mathbf{z}'}}(fz_1))\right) \\
             + \sum_{i=2}^{\dep(\ww{b})}E_{r-1}(\ww{b}^{(i)};\mathbf{z}')G_r(\ww{b}_{(i)};\mathbf{z}).
        \end{multlined}
    \end{align*}
    The same argument used to derive the inequalities in \eqref{eq-G_r-lower-bound1} and \eqref{eq-G_r-lower-bound2} shows that for $2\le i\le \dep(\ww{b})$,
    \[
    \ord_r(G_r(\ww{b}_{(i)};\mathbf{z}))>q^{r-1}\ge q
    \]
    and
    \[
    \ord_r\left(\sum_{1\ne f\in A_+}G_{b_1}^{\Lambda_{\mathbf{z}'}}(t_{\Lambda_{\mathbf{z}'}}f(z_1))\right)\ge q^{r-1}\ge q.
    \]
    On the other hand, by Proposition \ref{prop-basic-properties-for-Goss-polynomials} (ii), for $1\leq n\le q$, we have
    \[
    G_n^{\Lambda_{\mathbf{z}'}}(t_{\Lambda_{\mathbf{z}'}}(z_1))=t_{\Lambda_{\mathbf{z}'}}(z_1)^n.
    \]
    
    Recall our assumption that for each $i=1,\ldots,q$, the first entry of $\ww{a}_i$ is $i$.
    Therefore, applying the above discussion to $\ww{b}=\ww{a}_1,\ldots,\ww{a}_q$, we obtain that for $1\le i\le q-1$,
    \begin{equation}\label{eq-ai-expansion}
        E_r(\ww{a}_i;\mathbf{z})=E_{r-1}(\ww{a}_i;\mathbf{z})+E_{r-1}(\ww{a}_i^{(1)};\mathbf{z}')t_{\Lambda_{\mathbf{z}'}}(z_1)^{i}+t_{\Lambda_{\mathbf{z}'}}(z_1)^qF_i(t_{\Lambda_{\mathbf{z}'}}(z_1))
    \end{equation}
    and 
    \begin{equation}\label{eq-aq-expansion}
        E_r(\ww{a}_q;\mathbf{z})=E_{r-1}(\ww{a}_q;\mathbf{z})+t_{\Lambda_{\mathbf{z}'}}(z_1)^qF_q(t_{\Lambda_{\mathbf{z}'}}(z_1))
    \end{equation}
    for some $F_1,\ldots,F_q\in \mathcal{O}(\Omega^{r-1})[\![X]\!]$.

    Assume 
    \[
    h_1E_r(\ww{a}_1;\mathbf{z})+\cdots+h_qE_r(\ww{a}_q;\mathbf{z})=0
    \]
    for some $h_i\in\Ocal(\Omega^{r-1})$, $1\leq i\leq q$.
    By~\eqref{eq-ai-expansion} and~\eqref{eq-aq-expansion}, the above identity becomes
    \[
    \sum_{i=1}^{q}h_iE_{r-1}(\ww{a}_i;\mathbf{z}')+\sum_{i=1}^{q-1}h_iE_{r-1}(\ww{a}_i^{(1)};\mathbf{z}')t_{\Lambda_{\mathbf{z}'}}(z_1)^i+\left(\sum_{i=1}^{q}h_iF_i(t_{\Lambda_{\mathbf{z}'}}(z_1))\right)t_{\Lambda_{\mathbf{z}'}}(z_1)^q=0.
    \]
    By the uniqueness of Proposition \ref{prop-t-expansion}  for $t_{\Lambda_{\mathbf{z}'}}$-expansion, for $i=1,\ldots,q-1$, we have 
    \[
    h_iE_{r-1}(\ww{a}_i^{(1)};\mathbf{z}')=0,
    \]
    which implies that $h_i=0$, and hence $h_q=0$ as well.
    The proof is completed.
\end{proof}

\begin{rem}
    As a consequence of Proposition~\ref{prop-O(Ω^r-1)-linear-indep-of-MES}, we obtain that for $r\ge 2$, the functions $E_r(q;\bfz)$ and $E_r((1,q-1);\bfz)$ are $\Ocal(\Omega^{r-1})$-linearly independent.
    In particular, 
    \[
    E_r(q;\mathbf{z})-(\theta-\theta^q) E_r((1,q-1);\mathbf{z})\neq 0.
    \]
    Namely, Thakur's fundamental relation \cite[Theorem 5]{Tha09R1}
    \[
    \zeta_A(q) - (\theta - \theta^q)\zeta_A(1,q-1) =0
    \]
    cannot be lifted to the same identity among multiple Eisenstein series of rank $r\geq 2$.
\end{rem}

\section{Algebra structure on the \texorpdfstring{$q$}{q}-shuffle algebra \texorpdfstring{$\mathcal{E}$}{E}}\label{sec-Hopf-algebra-on-E}

In this section, we first utilize the machinery of multiple Eisenstein series to prove Theorem~\ref{thm-main-thm1}~(ii), which affirms Shi's associativity conjecture proposed in
\cite{Shi2018}.
Moreover, we prove Theorem~\ref{thm-main-thm2}, which consequently resolves the conjecture proposed in the authors' former paper \cite{CCHT2025}.

\subsection{Associativity of \texorpdfstring{$\mathcal{E}$}{E}}

Recall that for any $\FF_p$-subalgebra $L\subseteq\CC_\infty$, by Proposition \ref{prop-taking-constant-map}, the $L$-algebras $\{\Zcal^{(r)}_L\}_{r\in \NN}$ of multiple Eisenstein series form an inverse system with respect to $\pi_{r+1}:\mathcal{Z}^{(r+1)}_L\to \mathcal{Z}^{(r)}_L$, and hence $\varprojlim\limits_{r\in \mathbb{N}} \mathcal{Z}^{(r)}_L$ is an $L$-algebra.
We further recall that $(\Rcal,\ast)$ is defined in Definition~\ref{df-zeta-values-as-words}. 
The linear independence result established above enables us to embed $L\otimes_{\mathbb{F}_p} \mathcal{R}$ into $\varprojlim\limits_{{r\in \mathbb{N}}} \mathcal{Z}^{(r)}_L$ as in the following theorem.

\begin{thm}\label{thm-Rcal-inverse-limit}
    Let $L\subseteq \CC_\infty$ be an $\Fp$-subalgebra.
    We define
    \[
    \bfE_L:L\otimes_{\FF_p} \Rcal 
    \to 
    \varprojlim_{r\in \NN} \Zcal^{(r)}_L
    \]
    to be the unique $L$-linear map such that
    \[
    \bfE_L(1\otimes x_{\ww{a}}):=(E_r(\ww{a};\bfz))_{r\in \NN}.
    \]
    Then $\bfE_L$ is an injective $L$-algebra homomorphism.
\end{thm}
\begin{proof}
    Recall that for $r\ge 1$, the unique $\Fp$-linear map $\bfE_r : \Rcal \to \Zcal^{(r)}, x_{\ww{a}} \mapsto E_r(\ww{a};\bfz)$ given in Theorem~\ref{thm-E-hat} is an $\FF_p$-algebra homomorphism.
    It follows immediately that  $\bfE_L$ is an $L$-algebra homomorphism. 

    To prove the injectivity, suppose that $\bfE_L(\xi) = 0$ for some $\xi \in L\otimes_{\Fp} \Rcal$ and write
    \[
    \xi = \sum_{\ww{a}\in \Ical_{\le w}} c_{\ww{a}}\otimes x_{\ww{a}},
    \]
    where $c_{\ww{a}}\in L$  for all $\ww{a}\in \Ical_{\leq w}$ and $w\in \NN$.
    Now, we have
    \[
    0 = \bfE_L(\xi) = \left(\sum_{\ww{a}\in \Ical_{\leq w}} c_{\ww{a}} E_r(\ww{a};\bfz)\right)_{r\in \NN} \in \varprojlim_{r\in\NN}\Zcal_L^{(r)}.
    \]
    By Theorem~\ref{thm-linear-indep-of-MES}, for each $r>(w+1)+\log_q(w+1)$, the set
    \[
    \{E_r(\ww{a};\bfz): \ww{a}\in \Ical_{\le w}\}
    \]
    is a $\CC_\infty$-linearly independent set, which is in particular an $L$-linearly independent set.
    Therefore,  $c_{\ww{a}} = 0$ for all $\ww{a}\in \Ical$, which implies that $\xi =0$.
    The proof is completed.
\end{proof}

In particular, since $\Rcal$ embeds into the associative algebra $\varprojlim\limits_{r\in \NN} \Zcal^{(r)}$, we obtain the following corollary.

\begin{cor}\label{cor-assoc-of-R}
    $(\Rcal,\ast)$ forms a commutative associative $\mathbb{F}_p$-algebra.
\end{cor}

Recall that $(\Ecal,\ast)$ is given in Definition~\ref{df-Ecal} and $(\Rcal,\ast)$ forms an $\FF_p$-subalgebra of $(\Ecal,\ast)$.
We now aim to establish the associativity of $(\Ecal,\ast)$.
By Corollary~\ref{cor-assoc-of-R}, we can readily show that the associativity identity holds for a specific class of elements in $\Ecal$.

\begin{cor}\label{cor-R-ass} 
    For $\afk \in \Rcal$ and $\bfk,\cfk \in \Ecal$, we have $(\afk \ast \bfk) \ast \cfk = \afk \ast (\bfk \ast \cfk)$.
\end{cor}

\begin{proof}
    Let $\Ycal \sbe \Ecal$ be the $\Fp$-subspace spanned by $\{y_{\ww{a}} : \ww{a} \in \Ical \}$.
    We show that the associativity identity holds for the following cases.
	\begin{enumerate}
		\item $\afk,\bfk \in \Rcal$ and $\cfk \in \Ycal$.
		\item $\afk,\bfk \in \Rcal$ and $\cfk \in \Ecal$.
		\item $\afk \in \Rcal, \bfk \in \Ycal$ and $\cfk \in \Ecal$.
		\item $\afk \in \Rcal$ and $\bfk,\cfk \in \Ecal$.
	\end{enumerate}
	By distributive law on $(\mathcal{E},\ast)$, we may assume $\afk,\bfk,\cfk$ are words.
	
	(i) We have
	\[
	(x_{\ww{a}} \ast x_{\ww{b}}) \ast y_{\ww{c}}
	= x_{\ww{a}} \ast (x_{\ww{b}} y_{\ww{c}})
	= x_{\ww{a}} \ast (x_{\ww{b}} \ast y_{\ww{c}}),
	\]
    where the two equality follow from Definition~\ref{df-Ecal} (ii) and (v).
	
	(ii) We have
	\begin{align*}
		(x_{\ww{a}} \ast x_{\ww{b}}) \ast (x_{\ww{c}} y_{\ww{d}})
		&= (x_{\ww{a}} \ast x_{\ww{b}}) \ast (x_{\ww{c}} \ast y_{\ww{d}})
		\overset{(i)}{=} ((x_{\ww{a}} \ast x_{\ww{b}}) \ast x_{\ww{c}}) \ast y_{\ww{d}}  \\
		&= (x_{\ww{a}} \ast (x_{\ww{b}} \ast x_{\ww{c}})) \ast y_{\ww{d}}
		\overset{(i)}{=} x_{\ww{a}} \ast ((x_{\ww{b}} \ast x_{\ww{c}}) \ast y_{\ww{d}})
		= x_{\ww{a}} \ast (x_{\ww{b}} \ast (x_{\ww{c}} y_{\ww{d}})),
	\end{align*}
    where the first equality follows from Definition~\ref{df-Ecal} (ii), the third equality follows from Corollary~\ref{cor-assoc-of-R}, and the last equality follows from Definition~\ref{df-Ecal} (v).
	
	(iii) We have
	\[
	(x_{\ww{a}} \ast y_{\ww{b}}) \ast (x_{\ww{c}} y_{\ww{d}}) = (x_{\ww{a}}y_{\ww{b}}) * (x_{\ww{c}}y_{\ww{d}})
	= (x_{\ww{a}} \ast x_{\ww{c}}) \ast (y_{\ww{b}} \ast y_{\ww{d}})
	\overset{(ii)}{=} x_{\ww{a}} \ast (x_{\ww{c}} \ast (y_{\ww{b}} \ast y_{\ww{d}}))
	= x_{\ww{a}} \ast (y_{\ww{b}} \ast (x_{\ww{c}} y_{\ww{d}})),
	\]
    where the first equality follows from Definition~\ref{df-Ecal} (ii), and both the second equality and the last equality follow from Definition~\ref{df-Ecal} (v).
	
	(iv) We have
	\begin{align*}
		(x_{\ww{a}} \ast (x_{\ww{b}} y_{\ww{c}})) \ast (x_{\ww{d}} y_{\ww{e}})
		&= ((x_{\ww{a}} \ast x_{\ww{b}}) \ast y_{\ww{c}}) \ast (x_{\ww{d}} y_{\ww{e}})
		\overset{(iii)}{=} (x_{\ww{a}} \ast x_{\ww{b}}) \ast (y_{\ww{c}} \ast (x_{\ww{d}} y_{\ww{e}}))  \\
		&\overset{(ii)}{=} x_{\ww{a}} \ast (x_{\ww{b}} \ast (y_{\ww{c}} \ast (x_{\ww{d}} y_{\ww{e}})))  \\
		&\overset{(iii)}{=} x_{\ww{a}} \ast ((x_{\ww{b}} \ast y_{\ww{c}}) \ast (x_{\ww{d}} y_{\ww{e}}))
		= x_{\ww{a}} \ast ((x_{\ww{b}} y_{\ww{c}})\ast (x_{\ww{d}} y_{\ww{e}})),
	\end{align*}
    where the first equality follows from Definition~\ref{df-Ecal} (v) and the last equality follows from Definition~\ref{df-Ecal} (ii).
\end{proof}

We now apply Corollary~\ref{cor-R-ass} to prove that the map $\bfe:\Rcal\to \Ecal$ defined in Definition~\ref{defn-e-hat} is an $\Fp$-algebra homomorphism.
We mention that, assuming the associativity of $\Ecal$, the map $\bfe$ was already shown to be an $\Fp$-algebra homomorphism in \cite{CCHT2025}.
We first recall the following lemma.

\begin{lem}[{\cite[Lemma 3.6]{CCHT2025}}]\label{lem-(yb)*a=y(b*a)}
    For each $\afk\in \mathcal{R}$, $\bfk\in \mathcal{E}$ and $w\in\bbN$, we have 
    \[
    y_w(\bfk*\afk)=(y_w\bfk)*\afk.
    \]
\end{lem}

The following lemmas are strengthened versions of \cite[Lemma 3.8]{CCHT2025} and \cite[Proposition 3.9]{CCHT2025}.
While the authors in \cite{CCHT2025} proved the two identities on $\mathcal{E}$ modulo the associator ideal, Corollary~\ref{cor-R-ass} enables us to refine the arguments to show that these identities hold in $\mathcal{E}$ itself.
This leads to Lemmas~\ref{lem-yx*yx} and~\ref{lem-ye(x)*ye(x)}.
Note that we adopt the convention introduced in Remark~\ref{rem-first-entry-convention}.

\begin{lem}[cf. {\cite[Lemma 3.8]{CCHT2025}}]  \label{lem-yx*yx}
    For all indices $\ww{a},\ww{b},\ww{v},\ww{w}\in\Ical$, we have
    \begin{multline*}
        (y_{\ww{a}}x_{\ww{v}})*(y_{\ww{b}}x_{\ww{w}})
        = y_{a_1}\left((y_{\ww{a}^{(1)}}x_{\ww{v}})*(y_{\ww{b}}x_{\ww{w}}) \right)
        + y_{b_1} \left((y_{\ww{a}}x_{\ww{v}})*(y_{\ww{b}^{(1)}}x_{\ww{w}})\right) \\
        + y_{a_1+b_1} \left((y_{\ww{a}^{(1)}}x_{\ww{v}})*(y_{\ww{b}^{(1)}}x_{\ww{w}})\right)
        + \sum_{\substack{i+j=a_1+b_1 \\ q-1 \mid j}} \Delta^{i,j}_{a_1,b_1} y_i \left(\left((y_{\ww{a}^{(1)}}x_{\ww{v}})*(y_{\ww{b}^{(1)}}x_{\ww{w}}) \right)* \bfe(x_j) \right) .
    \end{multline*}
\end{lem}
\begin{proof}
    We first note that  $y_{\ww{a}} \ast y_{\ww{b}}$ can be written as
    \begin{multline}  \label{eq-rewrite-shuffle-of-y}
        y_{\ww{a}} * y_{\ww{b}} = y_{a_1} ( y_{\ww{a}^{(1)}} * y_{\ww{b}} ) 
        + y_{b_1} ( y_{\ww{a}} * y_{\ww{b}^{(1)}} ) 
        + y_{a_1+b_1} ( y_{\ww{a}^{(1)}} * y_{\ww{b}^{(1)}} ) \\
        + \sum_{\substack{i+j = a_1 + b_1 \\ q-1 \mid  j}} 
        \Delta^{i,j}_{a_1,b_1} y_i \left( ( y_{\ww{a}^{(1)}} * y_{\ww{b}^{(1)}} ) * \bfe(x_j) \right).
    \end{multline}
    Let $\mathfrak{a} := x_{\ww{v}}*x_{\ww{w}}$.
    Then we have
    \begin{align*}
        &(y_{\ww{a}}x_{\ww{v}})*(y_{\ww{b}}x_{\ww{w}})
        = (y_{\ww{a}}*y_{\ww{b}})*\mathfrak{a} \\
        ={} &\begin{multlined}[t]
            \left(y_{a_1}(y_{\ww{a}^{(1)}}*y_{\ww{b}})\right) * \mathfrak{a}
            + \left(y_{b_1}(y_{\ww{a}}*y_{\ww{b}^{(1)}})\right) * \mathfrak{a}
            + \left(y_{a_1+b_1}(y_{\ww{a}^{(1)}}*y_{\ww{b}^{(1)}})\right) * \mathfrak{a} \\
            + \sum_{\substack{i+j=a_1+b_1 \\ q-1\mid j}}\Delta^{i,j}_{a_1,b_1} \left(y_i \left((y_{\ww{a}^{(1)}}*y_{\ww{b}^{(1)}})*\bfe(x_j) \right) \right) * \mathfrak{a} \quad \text{(by \eqref{eq-rewrite-shuffle-of-y})}
        \end{multlined} \\
        ={} &\begin{multlined}[t]
            y_{a_1}\left((y_{\ww{a}^{(1)}}*y_{\ww{b}})*\mathfrak{a}\right)
            + y_{b_1}\left((y_{\ww{a}}*y_{\ww{a}^{(1)}})*\mathfrak{a}\right)
            + y_{a_1+b_1}\left((y_{\ww{a}^{(1)}}*y_{\ww{b}^{(1)}})*\mathfrak{a}\right) \\
            + \sum_{\substack{i+j=a_1+b_1 \\ q-1\mid j}} \Delta^{i,j}_{a_1,b_1} y_i \left(\left((y_{\ww{a}^{(1)}}*y_{\ww{b}^{(1)}})*\bfe(x_j)\right)*\mathfrak{a} \right) \quad\text{(by Lemma \ref{lem-(yb)*a=y(b*a)})}.
        \end{multlined}
    \end{align*}
    For the last term, we have
    \begin{align*}
        \left(\left(y_{\ww{a}^{(1)}}*y_{\ww{b}^{(1)}}\right)*\bfe(x_j)\right) * \mathfrak{a}
        &= (y_{\ww{a}^{(1)}}*y_{\ww{b}^{(1)}})*(\bfe(x_j)*\mathfrak{a}) \quad\text{(by Corollary \ref{cor-R-ass})} \\
        &=(y_{\ww{a}^{(1)}}*y_{\ww{b}^{(1)}})*(\mathfrak{a}*\bfe(x_j)) \\
        &= \left((y_{\ww{a}^{(1)}}*y_{\ww{b}^{(1)}})*\afk \right)*\bfe(x_j) \quad\text{(by Corollary \ref{cor-R-ass})}.
    \end{align*}
    Therefore,
    \begin{align*}
        &(y_{\ww{a}}x_{\ww{v}})*(y_{\ww{b}}x_{\ww{w}}) \\
        ={} &\begin{multlined}[t]
            y_{a_1}\left((y_{\ww{a}^{(1)}}*y_{\ww{b}})*\mathfrak{a}\right)
            + y_{b_1}\left((y_{\ww{a}}*y_{\ww{b}^{(1)}})*\mathfrak{a}\right)
            + y_{a_1+b_1}\left((y_{\ww{a}^{(1)}}*y_{\ww{b}^{(1)}})*\mathfrak{a}\right) \\
            + \sum_{\substack{i+j=a_1+b_1 \\ q-1\mid j}} \Delta^{i,j}_{a_1,b_1} y_i \left(\left((y_{\ww{a}^{(1)}}*y_{\ww{b}^{(1)}})*\afk \right)*\bfe(x_j)\right)
        \end{multlined}\\
        ={} &\begin{multlined}[t]
            y_{a_1}\left((y_{\ww{a}^{(1)}}x_{\ww{v}})*(y_{\ww{b}}x_{\ww{w}}) \right)
            + y_{b_1} \left((y_{\ww{a}}x_{\ww{v}})*(y_{\ww{b}^{(1)}}x_{\ww{w}})\right) + y_{a_1+b_1} \left((y_{\ww{a}^{(1)}}x_{\ww{v}})*(y_{\ww{b}^{(1)}}x_{\ww{w}})\right) \\
            + \sum_{\substack{i+j=a_1+b_1 \\ q-1 \mid j}} \Delta^{i,j}_{a_1,b_1} y_i \left(\left((y_{\ww{a}^{(1)}}x_{\ww{v}})*(y_{\ww{b}^{(1)}}x_{\ww{w}}) \right)* \bfe(x_j) \right).
        \end{multlined}
    \end{align*}
\end{proof}

\begin{lem}[cf. {\cite[Proposition 3.9]{CCHT2025}}]\label{lem-ye(x)*ye(x)}
    For each $a,b\in \mathbb{N}$ and indices $\ww{a},\ww{b}\in\Ical$, we have 
    \begin{multline*}
        \left(y_a\bfe(x_{\ww{a}})\right) * \left(y_b\bfe(x_{\ww{b}})\right)
        = y_a\left(\bfe(x_{\ww{a}})*\left(y_b\bfe(x_{\ww{b}})\right) \right)
        + y_b \left( \left(y_a\bfe(x_{\ww{a}}) \right) * \bfe(x_{\ww{b}}) \right)  \\
        +y_{a+b} \left(\bfe(x_{\ww{a}})*\bfe(x_{\ww{b}}) \right)
        +\sum_{\substack{i+j=a+b \\ q-1\mid j}} \Delta^{i,j}_{a,b} y_i \left( \left(\bfe(x_{\ww{a}}) * \bfe(x_{\ww{b}}) \right) * \bfe(x_j) \right).
    \end{multline*}
\end{lem}
\begin{proof}
    Notice that
    \begin{align*}
        &(y_a\bfe(x_{\ww{a}}))*(y_b\bfe(x_{\ww{b}}))  \\
        ={} &\left(\sum_{k=0}^{\dep(\ww{a})}y_ay_{\ww{a}_{(k)}}x_{\ww{a}^{(k)}}\right)*\left(\sum_{\ell=0}^{\dep(\ww{b})}y_by_{\ww{b}_{(\ell)}}x_{\ww{b}^{(\ell)}}\right) \quad\text{(by Definition \ref{defn-e-hat})} \\
        ={} &\sum_{\substack{0\leq k\leq \dep(\ww{a})\\0\leq \ell\leq \dep(\ww{b})}}(y_ay_{\ww{a}_{(k)}}x_{\ww{a}^{(k)}})*(y_by_{\ww{b}_{(\ell)}}x_{\ww{b}^{(\ell)}}) \\
        ={}  &\begin{multlined}[t]
            \sum_{\substack{0\leq k\leq \dep(\ww{a})\\0\leq \ell\leq \dep(\ww{b})}}y_a\left((y_{\ww{a}_{(k)}}x_{\ww{a}^{(k)}})*(y_by_{\ww{b}_{(\ell)}}x_{\ww{b}^{(\ell)}})\right)
            +\sum_{\substack{0\leq k\leq \dep(\ww{a})\\0\leq \ell\leq \dep(\ww{b})}}y_b \left((y_ay_{\ww{a}_{(k)}}x_{\ww{a}^{(k)}})*(y_{\ww{b}_{(\ell)}}x_{\ww{b}^{(\ell)}}) \right) \\
            +\sum_{\substack{0\leq k\leq \dep(\ww{a})\\0\leq \ell\leq \dep(\ww{b})}}y_{a+b}\left((y_{\ww{a}_{(k)}}x_{\ww{a}^{(k)}})*(y_{\ww{b}_{(\ell)}} x_{\ww{b}^{(\ell)}})\right)
            \\
            +\sum_{\substack{0\leq k\leq \dep(\ww{a})\\0\leq \ell\leq \dep(\ww{b})}}\sum_{\substack{i+j=a+b \\ q-1\mid j}}\Delta^{i,j}_{a,b}y_i\left(\left((y_{\ww{a}_{(k)}}x_{\ww{a}^{(k)}})*(y_{\ww{b}_{(\ell)}}x_{\ww{b}^{(\ell)}})\right)*\bfe(x_j) \right) \quad
            \text{(by Lemma \ref{lem-yx*yx})}
        \end{multlined} \\
        ={} &\begin{multlined}[t]
            y_a \left( \bfe(x_{\ww{a}}) * (y_b\bfe(x_{\ww{b}}) )\right) 
            + y_b \left( (y_a\bfe(x_{\ww{a}}) ) * \bfe(x_{\ww{b}}) \right)
            +y_{a+b}\left(\bfe(x_{\ww{a}})*\bfe(x_{\ww{b}})\right) \\ +\sum_{\substack{i+j=a+b \\ q-1\mid j}}\Delta^{i,j}_{a,b}y_i \left( \left( \bfe(x_{\ww{a}}) * \bfe(x_{\ww{b}}) \right) * \bfe(x_j) \right) \quad\text{(by Definition \ref{defn-e-hat})}.
        \end{multlined}
    \end{align*}
\end{proof}

With the necessary identities in hand, we are now ready to prove the following proposition using the argument in \cite[Proposition~3.7]{CCHT2025}.

\begin{prop}[cf. {\cite[Proposition 3.7]{CCHT2025}}] \label{prop-bfe-is-alg-hom}
    The map $\bfe:\mathcal{R}\to \mathcal{E}$ is an $\mathbb{F}_p$-algebra homomorphism.
\end{prop}
\begin{proof}
    First, notice that for all $a\in \mathbb{N}$ and non-empty index $\ww{a}\in\Ical$, we have
    \begin{equation} \label{eq-bfe-alg-hom-observation}
        \bfe(x_ax_{\ww{a}})=x_ax_{\ww{a}}+y_ax_{\ww{a}}+\sum_{i=1}^{\dep(\ww{a})}y_ax_{\ww{a}^{(i)}}y_{\ww{a}_{(i)}}=x_ax_{\ww{a}}+y_a\bfe(x_{\ww{a}}).
    \end{equation}
    We now claim that
    \[
    \bfe(x_{\ww{a}})*\bfe(x_{\ww{b}}) 
    = \bfe(x_{\ww{a}}*x_{\ww{b}})
    \]
    for any indices $\ww{a},\ww{b} \in \Ical$.
    If either $\dep(\ww{a})$ or $\dep(\ww{b})$ is zero, then $\bfe(x_{\ww{a}})=1$ or $\bfe(x_{\ww{b}})=1$, and the result is clear.
    Assuming both $\ww{a}$ and $\ww{b}$ are non-empty, we proceed by induction on the total depth $\dep(\ww{a})+\dep(\ww{b})$.
    For the base case $\dep({\ww{a}}) = \dep(\ww{b}) = 1$, we have 

    \begin{align*}
        \bfe(x_a*x_b)
        &=\begin{multlined}[t]
            x_a x_b+x_b x_a+x_{a+b}+\sum_{\substack{i+j=a+b \\ q-1\mid j}}\Delta^{i,j}_{a,b}x_ix_j+y_ax_b+y_bx_a+y_{a+b} \\
            +\sum_{\substack{i+j=a+b \\ q-1\mid j}}\Delta^{i,j}_{a,b}y_ix_j+y_ay_b+y_by_a+\sum_{\substack{i+j=a+b \\ q-1\mid j}}\Delta^{i,j}_{a,b}y_iy_j \\
        \end{multlined} \\
        &=(x_a*x_b)+(y_b*x_a)+(x_a*y_b)+(y_a*y_b) \\
        &=(x_a+y_a)*(x_b+y_b)  \\
        &=\bfe(x_a)*\bfe(x_b).
    \end{align*}
    
    Given $N>2$, suppose the claim holds for any indices with total depth less than $N$.
    Let $\ww{a},\ww{b}\in\Ical$ be non-empty indices with $\dep(\ww{a})+\dep(\ww{b})=N$.
    Then
    \begin{align}
        &\bfe(x_{\ww{a}}) * \bfe(x_{\ww{b}}) \notag \\
        ={} &\left( x_{\ww{a}}+y_{a_1} \bfe(x_{\ww{a}^{(1)}}) \right) * \left(x_{\ww{b}}+y_{b_1}\bfe(x_{\ww{b}^{(1)}}) \right) \quad \text{(by \eqref{eq-bfe-alg-hom-observation})} \notag \\
        ={} &x_{\ww{a}}*x_{\ww{b}}
        +\left( y_{a_1}\bfe(x_{\ww{a}^{(1)}}) \right) * x_{\ww{b}}
        +\left(y_{\ww{b}_1}\bfe(x_{\ww{b}^{(1)}})\right)*x_{\ww{a}}
        +\left( y_{a_1}\bfe(x_{\ww{a}^{(1)}}) \right) * \left(y_{b_1} \bfe(x_{\ww{b}^{(1)}}) \right). \label{eq-e-shu-e}
    \end{align}
    On the other hand,
    \begin{align}
        &\bfe(x_{\ww{a}} * x_{\ww{b}}) \notag \\
        ={} &\begin{multlined}[t]
            \bfe\left( x_{a_1}(x_{\ww{a}^{(1)}} * x_{\ww{b}}) \right)+\bfe\left( x_{b_1}(x_{\ww{a}}*x_{\ww{b}^{(1)}}) \right)
            +\bfe\left( x_{a_1+b_1}(x_{\ww{a}^{(1)}}*x_{\ww{b}^{(1)}}) \right) \\
            +\sum_{\substack{i+j=a_1+b_1 \\ q-1\mid j}}\Delta^{i,j}_{a_1,b_1}\bfe\left( x_i\left( (x_{\ww{a}^{(1)}}*x_{\ww{b}^{(1)}})*x_j \right) \right)
        \end{multlined} \notag \\
        ={} &\begin{multlined}[t]
            x_{a_1}(x_{\ww{a}^{(1)}}*x_{\ww{b}})+x_{b_1}(x_{\ww{a}}*x_{\ww{b}^{(1)}})+x_{a_1+b_1}(x_{\ww{a}^{(1)}}*x_{\ww{b}^{(1)}}) \\
            +\sum_{\substack{i+j=a_1+b_1 \\ q-1\mid j}}\Delta^{i,j}_{a_1,b_1}x_i\left( (x_{\ww{a}^{(1)}}*x_{\ww{b}^{(1)}})*x_j \right) \\ +y_{a_1}\bfe(x_{\ww{a}^{(1)}}*x_{\ww{b}})+y_{b_1}\bfe(x_{\ww{a}}*x_{\ww{b}^{(1)}})+y_{a_1+b_1}\bfe(x_{\ww{a}^{(1)}}*x_{\ww{b}^{(1)}}) \\
            +\sum_{\substack{i+j=a_1+b_1 \\ q-1\mid j}}\Delta^{i,j}_{a_1,b_1}y_i\bfe\left( (x_{\ww{a}^{(1)}}*x_{\ww{b}^{(1)}})*x_j \right)
            \quad\text{(by \eqref{eq-bfe-alg-hom-observation})}
        \end{multlined} \notag \\
        ={} &\begin{multlined}[t]
            x_{\ww{a}}*x_{\ww{b}} +y_{a_1}\left(\bfe(x_{\ww{a}^{(1)}})*\bfe(x_{\ww{b}})\right)+y_{b_1}\left(\bfe(x_{\ww{a}})*\bfe(x_{\ww{b}^{(1)}})\right)
            + y_{a_1+b_1}\left(\bfe(x_{\ww{a}^{(1)}})*\bfe(x_{\ww{b}^{(1)}})\right) \\
            +\sum_{\substack{i+j=a_1+b_1 \\ q-1\mid j}}\Delta^{i,j}_{a_1,b_1}y_i\left( \left( \bfe(x_{\ww{a}^{(1)}})*\bfe(x_{\ww{b}^{(1)}})\right) *\bfe(x_j) \right) \\
            \text{(by induction hypothesis)}
        \end{multlined} \notag \\
        ={} &\begin{multlined}[t]
            x_{\ww{a}}*x_{\ww{b}}
            +y_{a_1}\left( \bfe(x_{\ww{a}^{(1)}})* \left( x_{\ww{b}}+y_{b_1} \bfe(x_{\ww{b}^{(1)}}) \right) \right)
            +y_{b_1} \left(\bfe(x_{\ww{b}^{(1)}})*\left(x_{\ww{a}}+y_{a_1}\bfe(x_{\ww{a}^{(1)}}) \right)\right) \\
            +y_{a_1+b_1}\bfe(x_{\ww{a}^{(1)}}*x_{\ww{b}^{(1)}})
            +\sum_{\substack{i+j=a_1+b_1 \\ q-1\mid j}}\Delta^{i,j}_{a_1,b_1} y_i\left( \left( \bfe(x_{\ww{a}^{(1)}})*\bfe(x_{\ww{b}^{(1)}})\right) *\bfe(x_j) \right)
            \quad\text{(by \eqref{eq-bfe-alg-hom-observation})}
        \end{multlined} \notag \\
        ={} &\begin{multlined}[t]
            x_{\ww{a}}*x_{\ww{b}}
            +y_{a_1}\left( \bfe(x_{\ww{a}^{(1)}})* x_{\ww{b}}\right)+y_{a_1}\left(\bfe(x_{\ww{a}^{(1)}})*\left(y_{b_1} \bfe(x_{\ww{b}^{(1)}})\right)\right) 
            \\+y_{b_1} \left(\bfe(x_{\ww{b}^{(1)}})*x_{\ww{a}}\right)+y_{b_1}\left(\bfe(x_{\ww{b}^{(1)}})*\left(y_{a_1}\bfe(x_{\ww{a}^{(1)}})\right) \right)
            +y_{a_1+b_1}\bfe(x_{\ww{a}^{(1)}}*x_{\ww{b}^{(1)}}) \\
            +\sum_{\substack{i+j=a_1+b_1 \\ q-1\mid j}}\Delta^{i,j}_{a_1,b_1} y_i\left( \left( \bfe(x_{\ww{a}^{(1)}})*\bfe(x_{\ww{b}^{(1)}})\right) *\bfe(x_j) \right)
        \end{multlined} \notag \\
        ={} &\begin{multlined}[t]
            x_{\ww{a}}*x_{\ww{b}}
            +\left( y_{a_1} \bfe(x_{\ww{a}^{(1)}})\right) * x_{\ww{b}}
            +y_{a_1}\left(\bfe(x_{\ww{a}^{(1)}}) * \left( y_{b_1}\bfe(x_{\ww{b}^{(1)}})\right)\right) \\
            +\left(y_{b_1}\bfe(x_{\ww{b}^{(1)}})\right) * x_{\ww{a}}
            +y_{b_1} \left( \left( y_{a_1}\bfe(x_{\ww{a}^{(1)}}) \right) * \bfe(x_{\ww{b}^{(1)}}) \right)
            +y_{a_1+b_1}\bfe(x_{\ww{a}^{(1)}}*x_{\ww{b}^{(1)}}) \\
            +\sum_{\substack{i+j=a_1+b_1 \\ q-1\mid j}}\Delta^{i,j}_{a_1,b_1} y_i\left( \left( \bfe(x_{\ww{a}^{(1)}})*\bfe(x_{\ww{b}^{(1)}})\right) *\bfe(x_j) \right)
            \quad\text{(by Lemma \ref{lem-(yb)*a=y(b*a)})}.
        \end{multlined}\label{eq-e-x-shu-x}
    \end{align}
    Comparing the terms in \eqref{eq-e-shu-e} and \eqref{eq-e-x-shu-x}, we see that the identity
    \[
    \bfe(x_{\ww{a}} * x_{\ww{b}}) =\bfe(x_{\ww{a}}) * \bfe(x_{\ww{b}})
    \]
    is equivalent to
    \begin{align*}
        &\left( y_{a_1}\bfe(x_{\ww{a}^{(1)}})\right) * \left(y_{b_1}\bfe(x_{\ww{b}^{(1)}})\right) \\
        ={} &\begin{multlined}[t]
            y_{a_1}\left(\bfe(x_{\ww{a}^{(1)}}) * \left( y_{b_1}\bfe(x_{\ww{b}^{(1)}})\right)\right)
            +y_{b_1} \left( \left( y_{a_1}\bfe(x_{\ww{a}^{(1)}}) \right) * \bfe(x_{\ww{b}^{(1)}}) \right) \\
            +y_{a_1+b_1}\bfe(x_{\ww{a}^{(1)}}*x_{\ww{b}^{(1)}})
            +\sum_{\substack{i+j=a_1+b_1 \\ q-1\mid j}}\Delta^{i,j}_{a_1,b_1} y_i\left( \left( \bfe(x_{\ww{a}^{(1)}})*\bfe(x_{\ww{b}^{(1)}})\right) *\bfe(x_j) \right),
        \end{multlined}
    \end{align*}
    which follows from Lemma \ref{lem-ye(x)*ye(x)}. 
\end{proof}

Armed with Proposition~\ref{prop-bfe-is-alg-hom}, we proceed to prove the associativity of $\Ecal$ by exploiting its $\Rcal$-algebra structure.
For this purpose, we equip $\Ecal$ with an $\Rcal$-algebra structure via the natural inclusion map $\iota:\Rcal\to \Ecal$, and construct an $\Rcal$-basis of $\Ecal$ in the following lemma.

\begin{lem}\label{lem-bfe-basis}
    The set 
    \[
    \{\bfe(x_{\ww{a}}):\ww{a}\in\Ical\}
    \]
    forms an $\Rcal$-basis of $\Ecal$.
\end{lem}
\begin{proof}
    We first claim that $\{\bfe(x_{\ww{a}}): \ww{a}\in \Ical\}$ spans $\Ecal$ over $\Rcal$.
    It suffices to check for each $\ww{a}\in\Ical$, 
    \[
    y_{\ww{a}}\in\Span_{\Rcal}\{\bfe(x_{\ww{b}}):\ww{b}\in\Ical\}
    \]
    and we proceed by induction on $\dep(\ww{a})$.
    When $\dep(\ww{a})=0$, we have $\ww{a}=\varnothing$ and $\bfe(x_\varnothing)=1=y_{\varnothing}$.
    Given $N\geq 1$, assume that the result holds for all index $\ww{a}\in \Ical$ with $\dep(\ww{a})<N$.
    When $\dep(\ww{a})=N$, we note that by Definition \ref{defn-e-hat}, 
    \[
    \bfe(x_{\ww{a}})=y_{\ww{a}}+\sum_{i=0}^{N-1}x_{\ww{a}^{(i)}}\ast y_{\ww{a}_{(i)}},
    \]
    and that by induction hypothesis, 
    \[
    \sum_{i=0}^{N-1}x_{\ww{a}^{(i)}}\ast y_{\ww{a}_{(i)}}\in\Span_{\Rcal}\{\bfe(x_{\ww{b}}):\ww{b}\in\Ical\}.
    \]
    It follows immediately that $y_{\ww{a}}\in\Span_{\Rcal}\{\bfe(x_{\ww{b}}):\ww{b}\in\Ical\}$. 

    Now, we prove the $\Rcal$-linear independence of $\{\bfe(x_{\ww{a}}):\ww{a}\in\Ical\}$.
    It suffices to show that
    \[
    \{\bfe(x_{\ww{b}}) : \ww{b}\in \Ical, \dep(\ww{b})\le N\}
    \]
    is $\Rcal$-linearly independent for all $N\in \NN$, and we proceed by induction on $N$.
    For the base case $N = 0$, we have
    \[
    \{\bfe(x_{\ww{b}}) : \ww{b}\in \Ical, \dep(\ww{b})\le N\} = \{1\},
    \]
    which is clearly $\Rcal$-linearly independent.
    Given $N\ge 1$, assume that the result holds for $\dep(\ww{b})\le N-1$ and suppose
    \begin{equation}\label{eq-ehat-lin-indep}
        \sum_{\substack{\ww{b}\in \Ical \\ \dep(\ww{b})\le N}} \afk_{\ww{b}} * \bfe(x_{\ww{b}}) = 0,
    \end{equation}
    where $\afk_{\ww{b}}\in \Rcal$ and $\afk_{\ww{b}} = 0$ for all but finitely many $\ww{b}$.
    By Definition \ref{defn-e-hat}, we notice that
    \[
    \sum_{\substack{\ww{b}\in \Ical \\ \dep(\ww{b})\le N}} \afk_{\ww{b}} * \left(y_{\ww{b}} + \sum_{i=0}^{\dep(\ww{b})-1} x_{\ww{b}^{(i)}}*y_{\ww{b}_{(i)}}\right) = \sum_{\substack{\ww{b}\in \Ical \\ \dep(\ww{b})= N}} \afk_{\ww{b}} * y_{\ww{b}} + \sum_{\substack{\ww{c}\in\Ical\\ \dep(\ww{c})<N}}\bfk_{\ww{c}}\ast y_{\ww{c}}=0
    \]
    for some $\bfk_{\ww{c}}\in \Rcal$.
    Therefore, by $\Rcal$-linearly independence of $\{y_{\ww{a}} : \ww{a}\in \Ical\}$, we have $\afk_{\ww{b}} = 0$ for all $\ww{b}\in \Ical$ with $\dep(\ww{b}) = N$.
    Hence, \eqref{eq-ehat-lin-indep} becomes
    \[
    \sum_{\substack{\ww{b}\in \Ical \\ \dep(\ww{b})\le N-1}} \afk_{\ww{b}} * \bfe(x_{\ww{b}}) = 0.
    \]
    By induction hypothesis, $\afk_{\ww{b}} = 0$ for all $\ww{b} \in \Ical$ with $\dep(\ww{b})\le N-1$.
    This completes the proof.
\end{proof}

\begin{thm}\label{thm-assoc-of-E}
    Let $\phi:\mathcal{R}\otimes_{\mathbb{F}_p} \Rcal\to \Ecal$ to be the unique $\bbF_p$-linear map such that
    \[
    \phi(x_{\ww{a}}\otimes x_{\ww{b}})=x_{\ww{a}} * \bfe(x_{\ww{b}}).
    \]
    Then $\phi$ is an $\mathbb{F}_p$-algebra isomorphism.
    In particular, $(\mathcal{E},\ast)$ is an commutative associative $\mathbb{F}_p$-algebra.
\end{thm}
\begin{proof}
    We first claim that $\mathscr{B}=\{x_{\ww{a}} * \bfe(x_{\ww{b}}):\ww{a},\ww{b}\in \Ical\}$ forms an $\mathbb{F}_p$-basis of $\mathcal{E}$.
    By Lemma \ref{lem-bfe-basis}, we know that $\{\bfe(x_{\ww{a}}):\ww{a}\in \mathcal{I}\}$ forms an $\mathcal{R}$-basis of $\mathcal{E}$. 
    Thus, $\mathscr{B}$ spans $\mathcal{E}$ over $\mathbb{F}_p$ since $\{x_{\ww{a}}:\ww{a}\in \Ical\}$ spans $\mathcal{R}$ over $\mathbb{F}_p$ and $\{\bfe(x_{\ww{b}}):\ww{b}\in \Ical\}$ spans $\mathcal{E}$ over $\Rcal$.
    To prove the $\FF_p$-linear independence of $\mathscr{B}$, we assume that 
    \[
    \sum_{\ww{a},\ww{b}\in\Ical}\epsilon_{\ww{a},\ww{b}}(x_{\ww{a}} * \bfe(x_{\ww{b}}))=0,
    \]
    where $\epsilon_{\ww{a},\ww{b}}\in \mathbb{F}_p$ and  $\epsilon_{\ww{a},\ww{b}}=0$ for all but finitely many $(\ww{a},\ww{b})\in\Ical^2$.
    Then we write
    \[
    \sum_{\ww{b}\in\Ical}\left(\sum_{\ww{a}\in\Ical}\epsilon_{\ww{a},\ww{b}}x_{\ww{a}}\right)*\bfe(x_{\ww{b}})=0.
    \]
    By Lemma \ref{lem-bfe-basis}, for any index $\ww{b}\in\Ical$, we have 
    \[
    \sum_{\ww{a}\in\Ical}\epsilon_{\ww{a},\ww{b}}x_{\ww{a}}=0.
    \]
    Since $\{x_{\ww{a}}:\ww{a}\in \Ical\}$ forms an $\mathbb{F}_p$-basis of $\mathcal{R}$, we see that $\epsilon_{\ww{a},\ww{b}}=0$ for all $\ww{a},\ww{b}\in\Ical$ and $\mathscr{B}$ is $\mathbb{F}_p$-linearly independent.
    Therefore, $\mathscr{B}$ is an $\mathbb{F}_p$-basis of $\Ecal$.

    As $\{x_{\ww{a}}\otimes x_{\ww{b}} : \ww{a}, \ww{b}\in \Ical\}$ and $\mathscr{B}$ are $\mathbb{F}_p$-bases of $\Rcal\otimes_{\FF_p} \Rcal$ and $\mathcal{E}$, respectively, it follows immediately that $\phi$ is a well-defined isomorphism of $\bbF_p$-vector spaces.
    Hence, it remains to show that $\phi$ is an $\Fp$-algebra homomorphism.
    Given indices $\ww{a},\ww{b},\ww{c},\ww{d}\in\Ical$, we have
    \begin{align*}
        \phi((x_{\ww{a}} \otimes x_{\ww{b}}) * (x_{\ww{c}} \otimes x_{\ww{d}})) &= \phi((x_{\ww{a}} * x_{\ww{c}}) \otimes (x_{\ww{b}} * x_{\ww{d}})) \\
        &= (x_{\ww{a}} * x_{\ww{c}}) * \bfe(x_{\ww{b}} * x_{\ww{d}})  \\
        &= (x_{\ww{a}} * x_{\ww{c}}) * (\bfe(x_{\ww{b}}) * \bfe(x_{\ww{d}})) \quad \text{(by Proposition \ref{prop-bfe-is-alg-hom})} \\
        &=((x_{\ww{a}} * x_{\ww{c}}) * \bfe(x_{\ww{b}})) * \bfe(x_{\ww{d}}) \quad \text{(by Corollary \ref{cor-R-ass})} \\
        &=(x_{\ww{a}} * (x_{\ww{c}} * \bfe(x_{\ww{b}}))) * \bfe(x_{\ww{d}}) \quad \text{(by Corollary \ref{cor-R-ass})} \\
        &=(x_{\ww{a}} * (\bfe(x_{\ww{b}}) * x_{\ww{c}})) * \bfe(x_{\ww{d}}) \\
        &=((x_{\ww{a}} * \bfe(x_{\ww{b}})) * x_{\ww{c}}) * \bfe(x_{\ww{d}}) \quad\text{(by Corollary \ref{cor-R-ass})} \\
        &=(x_{\ww{a}} * \bfe(x_{\ww{b}})) * (x_{\ww{c}} * \bfe(x_{\ww{d}})) \quad\text{(by Corollary \ref{cor-R-ass})} \\
        &=\phi(x_{\ww{a}}\otimes x_{\ww{b}}) * \phi(x_{\ww{c}} \otimes x_{\ww{d}}).
    \end{align*}
    Since $\mathcal{E}$ is isomorphic to $\mathcal{R}\otimes_{\bbF_p} \Rcal$ and the latter one is a commutative associative $\Fp$-algebra by Corollary \ref{cor-assoc-of-R}, $\Ecal$ is a commutative associative $\Fp$-algebra. The proof is completed.
\end{proof}

\subsection{Further discussion of \texorpdfstring{$\Ecal$}{E}}

In this section, we prove that any Hopf algebra structure on $\Rcal$ naturally induces a corresponding Hopf algebra structure on $\Ecal$.
We start with a review of the definition of Hopf algebras.

\begin{df}[Hopf $\mathcal{A}$-algebra]\label{df-Hopf-A-algebra}
    Let $\mathcal{A}$ be a commutative ring with unity.
    A \textit{Hopf $\mathcal{A}$-algebra} is a tuple $(\Bcal,\ast,u,\Delta,\varepsilon,S)$ consisting of a commutative associative $\Acal$-algebra $(\Bcal,\ast,u)$ and three $\mathcal{A}$-algebra homomorphisms such that the following diagrams commute:
    \begin{enumerate}
        \item a coproduct $\Delta:\mathcal{B}\to \mathcal{B}\otimes_{\mathcal{A}} \mathcal{B}$ satisfying the following coassociativity diagram
        \[
        \begin{tikzcd}
        	{\Bcal\otimes_{\Acal}\Bcal\otimes_{\Acal}\Bcal} & {\Bcal\otimes_{\Acal}\Bcal} \\
        	{\Bcal\otimes_{\Acal}\Bcal} & \Bcal
        	\arrow["{\id\otimes\Delta}"', from=1-2, to=1-1]
        	\arrow["{\Delta\otimes \id}", from=2-1, to=1-1]
        	\arrow["\Delta"', from=2-2, to=1-2]
        	\arrow["\Delta", from=2-2, to=2-1]
        \end{tikzcd}
        \]
        \item a counit $\varepsilon:\mathcal{B}\to \mathcal{A}$ satisfying
        \[
        \begin{tikzcd}
        	{\Bcal\otimes_{\Acal}\Acal} & {\Bcal\otimes_{\Acal}\Bcal} & {\Acal\otimes_{\Acal}\Bcal} \\
        	& \Bcal
        	\arrow["{\id\otimes\varepsilon}"', from=1-2, to=1-1]
        	\arrow["{\varepsilon\otimes\id}", from=1-2, to=1-3]
        	\arrow["\sim"', from=2-2, to=1-1, sloped]
        	\arrow["\Delta"', from=2-2, to=1-2]
        	\arrow["\sim"', from=2-2, to=1-3, sloped]
        \end{tikzcd}
        \]
        \item an antipode $S:\mathcal{B}\to \mathcal{B}$ satisfying
        \[
        \begin{tikzcd}
        	& {\Bcal\otimes_{\Acal}\Bcal} && {\Bcal\otimes_{\Acal}\Bcal} \\
        	\Bcal && \Acal && \Bcal \\
        	& {\Bcal\otimes_{\Acal}\Bcal} && {\Bcal\otimes_{\Acal}\Bcal}
        	\arrow["{S\otimes \id}", from=1-2, to=1-4]
        	\arrow["\ast", from=1-4, to=2-5]
        	\arrow["\Delta", from=2-1, to=1-2]
        	\arrow["\varepsilon", from=2-1, to=2-3]
        	\arrow["\Delta"', from=2-1, to=3-2]
        	\arrow["u", from=2-3, to=2-5]
        	\arrow["{\id\otimes S}"', from=3-2, to=3-4]
        	\arrow["\ast"', from=3-4, to=2-5]
        \end{tikzcd}
        \]
    \end{enumerate}
    In this case, $(\Delta,\varepsilon,S)$ is called a \textit{Hopf $\Acal$-algebra structure} on $(\Bcal,\ast,u)$.
\end{df}
\begin{df}[Hopf $\mathcal{A}$-algebra homomorphism]\label{df-Hopf-A-algebra-homomorphism}
    Let $(\Bcal,*,u,\Delta,\varepsilon,S)$ and $(\Ccal,*',u',\Delta',\varepsilon',\allowbreak S')$ be two Hopf $\Acal$-algebras.
    A \textit{Hopf $\mathcal{A}$-algebra homomorphism from $\Bcal$ to $\Ccal$} is an $\Acal$-algebra homomorphism $\psi:\Bcal\to \Ccal$ such that the following diagram commutes:
    \[
    \begin{tikzcd}
    	{\Bcal\otimes_{\Acal}\Bcal} & {\Ccal\otimes_{\Acal}\Ccal} \\
    	\Bcal & \Ccal
    	\arrow["{\psi\otimes\psi}", from=1-1, to=1-2]
    	\arrow["\Delta", from=2-1, to=1-1]
    	\arrow["\psi"', from=2-1, to=2-2]
    	\arrow["{\Delta'}"', from=2-2, to=1-2]
    \end{tikzcd}
    \]
\end{df}
\begin{rem}
    If $\psi: \Bcal \to \Ccal$ is a Hopf $\Acal$-algebra homomorphism, then the following diagrams commute
    \[
    \begin{tikzcd}
    	\Bcal && \Ccal \\
    	& \Acal
    	\arrow["\psi", from=1-1, to=1-3]
    	\arrow["\varepsilon"', from=1-1, to=2-2]
    	\arrow["{\varepsilon'}", from=1-3, to=2-2]
    \end{tikzcd}
    \quad
    \text{and}
    \quad
    \begin{tikzcd}
    	\Bcal & \Ccal \\
    	\Bcal & \Ccal
    	\arrow["\psi", from=1-1, to=1-2]
    	\arrow["S"', from=1-1, to=2-1]
    	\arrow["{S'}", from=1-2, to=2-2]
    	\arrow["\psi"', from=2-1, to=2-2]
    \end{tikzcd}
    \]
\end{rem}

We have the following theorem.

\begin{thm}\label{thm-unique-Hopf-alg-structure-extend-R}
    Let $(\Delta,\varepsilon,S)$ be any Hopf $\mathbb{F}_p$-algebra structure on $(\mathcal{R},*)$. There is a unique Hopf $\mathbb{F}_p$-algebra structure $(\Delta^{\mathrm{e}},\varepsilon^{\mathrm{e}},S^{\mathrm{e}})$ on $(\mathcal{E},\ast)$ such that the inclusion map $\iota:\mathcal{R}\hookrightarrow \mathcal{E}$ and $\bfe:\mathcal{R}\to \mathcal{E}$ are Hopf $\mathbb{F}_p$-algebra homomorphisms.
     Moreover, endow the Hopf $\mathbb{F}_p$-algebra structure on $\mathcal{R}\otimes_{\mathbb{F}_p} \mathcal{R}$ via the structure induced by the product $\Spec \mathcal{R}\times_{\mathbb{F}_p} \Spec \mathcal{R}$ of $\mathbb{F}_p$-group schemes.
        Then the unique $\mathbb{F}_p$-linear map
        \[
        \phi:\mathcal{R}\otimes_{\mathbb{F}_p} \mathcal{R}\to \Ecal,\quad x_{\ww{a}}\otimes x_{\ww{b}}\mapsto x_{\ww{a}}*\bfe(x_{\ww{b}})
        \]
        defined in Theorem~\ref{thm-assoc-of-E} becomes an isomorphism of Hopf $\mathbb{F}_p$-algebras.
        In other words, we have
        \[
        \Spec \Rcal \times_{\FF_p} \Spec \Rcal \simeq \Spec \Ecal
        \]
        as $\FF_p$-group schemes.
\end{thm}

\begin{proof}
    Fix a Hopf $\bbF_p$-algebra structure $(\Delta,\varepsilon,S)$ on $(\Rcal,\ast)$, we obtain an $\FF_p$-group scheme structure on $\Spec \Rcal$.
    This induces a natural $\FF_p$-group scheme structure on the product $\Spec \Rcal\times_{\FF_p} \Spec \Rcal$ and hence $\Rcal\otimes_{\FF_p} \Rcal$ is endowed an Hopf $\FF_p$-algebra structure $(\Tilde{\Delta},\Tilde{\varepsilon},\Tilde{S})$. 
    Moreover, let $\operatorname{pr}_1,\operatorname{pr}_2:\Spec \Rcal\times_{\FF_p} \Spec \Rcal\to \Spec \Rcal$ be the first (resp. second) projection.
    Then $\operatorname{pr}_1,\operatorname{pr}_2$ are both $\FF_p$-group scheme homomorphisms.
    In other words, $i_1,i_2:\Rcal\to \Rcal\otimes_{\FF_p} \Rcal$ defined by $i_1(\afk)=\afk\otimes 1$ and $i_2(\bfk)=1\otimes \bfk$ are both Hopf $\FF_p$-algebra homomorphisms.

    For the existence part, by Theorem \ref{thm-assoc-of-E}, the map $\phi:\Rcal\otimes_{\FF_p} \Rcal\to \Ecal$ defined by $\phi(x_{\ww{a}}\otimes x_{\ww{b}})=x_{\ww{a}} * \bfe(x_{\ww{b}})$ is an isomorphism of $\FF_p$-algebras, so $\phi$ induces a Hopf $\FF_p$-algebra structure 
    \[
    (\Delta^{\mathrm{e}},\varepsilon^{\mathrm{e}},S^{\mathrm{e}}):=((\phi\otimes \phi)\circ \Tilde{\Delta}\circ \phi^{-1}, \varepsilon\circ \phi^{-1}, \phi\circ \Tilde{S}\circ \phi^{-1})
    \]
    on $(\Ecal,\ast)$ and $\phi$ becomes a Hopf $\FF_p$-algebra isomorphism.
    Since $\iota=\phi\circ i_1$ and $\bfe=\phi\circ i_2$, $\iota$ and $\bfe$ are Hopf $\FF_p$-algebra homomorphism.

    Finally, we prove the uniqueness part.
    Assume that $(\Delta',\varepsilon',S')$ is another Hopf $\FF_p$-algebra structure on $(\Ecal,\ast)$ such that $\iota,\bfe:\Rcal\to \Ecal$ are both Hopf $\FF_p$-algebra homomorphisms.
    Since the tensor product is the coproduct in the category of Hopf $\FF_p$-algebras, there is a unique Hopf $\FF_p$-algebra homomorphism $\psi:\Rcal\otimes_{\FF_p} \Rcal\to \Ecal$ such that $\iota=\psi\circ i_1$ and $\bfe=\psi\circ i_2$.
    On the other hand, in the category of commutative associative $\FF_p$-algebra, we have $\iota=\phi\circ i_1$ and $\bfe=\phi\circ i_2$.
    By the universal property of the tensor product, it follows that $\psi=\phi$, which is an isomorphism of $\FF_p$-algebras by Theorem \ref{thm-assoc-of-E}.
    Thus, $\psi=\phi$ is a Hopf $\Fp$-algebra isomorphism.
    We obtain
    \begin{align*}
        (\Delta',\varepsilon',S') 
        &=((\psi\otimes \psi)\circ \Tilde{\Delta}\circ \psi^{-1}, \varepsilon\circ \psi^{-1}, \psi\circ \Tilde{S}\circ \psi^{-1}) \\
        &=((\phi\otimes \phi)\circ \Tilde{\Delta}\circ \phi^{-1}, \varepsilon\circ \phi^{-1}, \phi\circ \Tilde{S}\circ \phi^{-1})=(\Delta^{\mathrm{e}},\varepsilon^{\mathrm{e}},S^{\mathrm{e}})
    \end{align*}
    and the proof is completed.
\end{proof}

\begin{rem}
    Let $(\Delta, \varepsilon, S)$ be a fixed Hopf $\FF_p$-algebra structure on $\Rcal$ and unique Hopf $\mathbb{F}_p$-algebra structure $(\Delta^{\mathrm{e}},\varepsilon^{\mathrm{e}},S^{\mathrm{e}})$ on $(\mathcal{E},\ast)$ such that the inclusion map $\iota:\mathcal{R}\hookrightarrow \mathcal{E}$ and $\bfe:\mathcal{R}\to \mathcal{E}$ are Hopf $\mathbb{F}_p$-algebra homomorphisms.
    We may write down the maps $\Delta^{\mathrm{e}}$, $\varepsilon^{\mathrm{e}}$, and $S^{\mathrm{e}}$ explicitly in terms of $\Delta, \varepsilon$ and $S$ as follows.
    \begin{itemize}
        \item $\Delta^{\mathrm{e}}: \Ecal \to \Ecal \otimes_{\FF_p} \Ecal$ is given by the following. 
        For any indices $\ww{a},\ww{b}\in \Ical$,
        
        \begin{align*}
            &\Delta^{\mathrm{e}}(x_{\ww{a}}) = \Delta(x_{\ww{a}}) \\
            &\Delta^{\mathrm{e}}(y_{\ww{b}})
            = (\bfe\otimes \bfe)(\Delta(x_{\ww{b}})) -\sum_{i=0}^{\dep(\ww{a})-1}\Delta(x_{\ww{b}^{(i)}})*\Delta^{\mathrm{e}}( y_{\ww{b}_{(i)}}). \\
            &\Delta^{\mathrm{e}}(x_{\ww{a}} y_{\ww{b}})
            = \Delta(x_{\ww{a}}) \ast \Delta^{\mathrm{e}}(y_{\ww{b}}). 
        \end{align*}

        \item $\varepsilon: \Ecal \to \FF_p$ is given by the following.
        For any indices $\ww{a},\ww{b}\in \Ical$,
        
        \begin{align*}
            &\varepsilon^{\mathrm{e}}(x_{\ww{a}})=\varepsilon(x_{\ww{a}}). \\
            &\varepsilon^{\mathrm{e}}(y_{\ww{b}})=\varepsilon(x_{\ww{b}})-\sum_{i=0}^{\dep(\ww{b})-1}\varepsilon(x_{\ww{b}^{(i)}})\varepsilon^{\mathrm{e}}(y_{\ww{b}_{(i)}}). \\
            &\varepsilon^{\mathrm{e}}(x_{\ww{a}}y_{\ww{b}})=\varepsilon(x_{\ww{a}})\varepsilon^{\mathrm{e}}(y_{\ww{b}}).
        \end{align*}

        \item $S: \Ecal \to \Ecal$ is given by the following.
        For any indices $\ww{a},\ww{b}\in \Ical$,
        
        \begin{align*}
            &S^{\mathrm{e}}(x_{\ww{a}})=S(x_{\ww{a}}). \\
            &S^{\mathrm{e}}(y_{\ww{b}})
            = \bfe(S(x_{\ww{b}}))-\sum_{i=0}^{\dep(\ww{b})-1} S(x_{\ww{b}^{(i)}}) * S^{\mathrm{e}}(y_{\ww{b}_{(i)}}). \\
            &S^{\mathrm{e}}(x_{\ww{a}}y_{\ww{b}}) = S(x_{\ww{a}}) * S^{\mathrm{e}}(y_{\ww{b}}).
        \end{align*}
    \end{itemize}
\end{rem}

\section*{Acknowledgments}

The authors are grateful to Chieh-Yu Chang for suggesting this project and for his meticulous reading of the manuscript. His insightful remarks significantly enhanced the development of this work.
We also acknowledge the financial support from the National Science and Technology Council (NSTC) over the past few years.
\sloppy
\printbibliography

@article{Chen2015,
	author = {Huei-Jeng Chen},
	title = {On shuffle of double zeta values over {$\mathbb{F}_q[t]$}},
	journal = {J. Number Theory},
	year = {2015},
	volume = {148},
	pages = {153--163},
}

@article{Chen2017,
	author = {Huei-Jeng Chen},
	title = {On shuffle of double Eisenstein series in positive characteristic},
	journal = {J. Th{\'e}or. Nombres Bordeaux},
	year = {2017},
	volume = {29},
	number = {3},
	pages = {815--825},
}

@article{Gos1980,
	author = {David Goss},
	title = {The algebraist's upper half-plane},
	journal = {Bull. Amer. Math. Soc. (N.S.)},
	year = {1980},
	volume = {2},
	number = {3},
	pages = {391--415},
}

@article{Tha09R1,
    author = {Dinesh S. Thakur},
    title = {Relations between Multizeta Values for $\mathbb{F}_q[t]$},
    journal = {Int. Math. Res. Not. IMRN},
    year = {2009},
    volume = {2009},
    number = {12},
    pages = {2318--2346},
}

@article{thakur2010shuffle,
	author = {Dinesh S. Thakur},
	title = {Shuffle relations for function field multizeta values},
	journal = {Int. Math. Res. Not. IMRN},
	year = {2010},
	volume = {2010},
	number = {11},
	pages = {1973--1980},
}

@phdthesis{Shi2018,
	author = {Shuhui Shi},
	title = {Multiple zeta values over {$\mathbb{F}_q[t]$}},
	school = {University of Rochester},
	year = {2018},
}

@book{thakur2004function,
	title={Function Field Arithmetic},
	author={Dinesh S. Thakur},
	year={2004},
	publisher={World Scientific Publishing Co. Pte. Ltd.},
}

@article{Drinfeld1974,
	author    = {Vladimir G. Drinfeld},
	title     = {Elliptic modules},
	journal   = {Math. USSR Sb.},
	volume    = {23},
	number    = {4},
	pages     = {561--592},
	year      = {1974},
}

@article{bt2018double,
	author    = {Henrik Bachmann and Koji Tasaka},
	title     = {The double shuffle relations for multiple Eisenstein series},
	journal   = {Nagoya Math. J.},
	volume    = {230},
	pages     = {180--212},
	year      = {2018},
}

@mastersthesis{bachmann2012multiple,
	author = {Henrik Bachmann},
	title = {Multiple Zeta-Werte und die Verbindung zu Modulformen durch multiple Eisensteinreihen},
	school = {Universität Hamburg},
	year = {2012},
}

@inproceedings{gkz2006double,
	author = {Herbert Gangl and Masanobu Kaneko and Don Zagier},
	title = {Double zeta values and modular forms},
	booktitle = {Automorphic Forms and Zeta Functions},
	pages = {71--106},
	year = {2006},
	publisher = {World Scientific},
}

@book{goss1996basic,
	title={Basic Structures of Function Field Arithmetic},
	author={David Goss},
	series={Ergeb. Math. Grenzgeb. (3)},
	volume={35},
	year={1996},
	publisher={Springer-Verlag Berlin Heidelberg}
}

@article{gekeler1988coefficients,
	title = {On the coefficients of Drinfeld modular forms},
	author = {Ernst-Ulrich Gekeler},
	pages = {667--700},
	volume = {93},
	journal={Invent. Math.},
	year = {1988},
}

@article{bbp2024drinfeld,
  title={Drinfeld Modular Forms of Arbitrary Rank},
  author={Dirk Basson and Florian Breuer and Richard Pink},
  journal={Mem. Amer. Math. Soc.},
  year={2024},
  volume={304},
  number={1531},
}

@misc{ikldp2023hopf,
	title={Hopf algebras and multiple zeta values in positive characteristic},
	author={Bo-Hae Im and Hojin Kim and Khac Nhuan Le and Ngo Dac, Tuan and Lan Huong Pham},
	year={2023},
	note={arXiv preprint, arXiv:2301.05906},
}

@misc{CCHT2025,
    title={On $q$-Shuffle Relations for Multiple Eisenstein Series of Arbitrary Rank in Positive Characteristic}, 
    author={Ting-Wei Chang and Song-Yun Chen and Fei-Jun Huang and Hung-Chun Tsui},
    year={2025},
    note={arXiv preprint, arXiv:2504.18879},
}

@article{Pel2025,
    author    = {Federico Pellarin},
    title     = {The Analytic Theory of Vectorial Drinfeld Modular Forms},
    journal   = {Mem. Amer. Math. Soc.},
    volume    = {312},
    year      = {2025},
    number    = {1581},
}

@article{Gek25-DMF-VII,
    title={On Drinfeld modular forms of higher rank VII: Expansions at the boundary}, 
    author={Ernst-Ulrich Gekeler},
    year={2025},
    journal = {J. Number Theory},
    volume  = {269},
    pages   = {260--340},
    year    = {2025},
}

@misc{bachmann2026relations,
  author       = {Henrik Bachmann and Hayato Kanno},
  title        = {Relations and Derivatives of Multiple Eisenstein Series},
  year         = {2026},
  note         = {arXiv preprint, arXiv:2602.08176},
}

\end{document}